\documentclass{article}
\usepackage[T1]{fontenc}
\usepackage{aeguill}
\usepackage{mathrsfs}
\usepackage{latexsym,amssymb,amsfonts,amsmath}
\usepackage[pdftex]{color}
\usepackage{graphicx}%in case of figures
\usepackage{amsmath,amsthm,amssymb}
\usepackage{color}
\usepackage{bm}
\usepackage{comment}
\usepackage[pagewise]{lineno}%\linenumbers

\providecommand{\abs}[1]{\lvert#1\rvert}%absolute value
\providecommand{\norm}[1]{\lVert#1\rVert}%norm

\newtheorem{theoreme}{Theorem}[section]
\newtheorem{proposition}[theoreme]{Proposition}

\newtheorem{definition}[theoreme]{Definition}
\newtheorem{remarque}[theoreme]{Remark}

\numberwithin{equation}{section}

\usepackage{tikz}

\usepackage[english]{babel}

\usepackage[letterpaper,top=2cm,bottom=2cm,left=1.5cm,right=2cm,marginparwidth=1.75cm]{geometry}

\usepackage{amsmath}
\usepackage{graphicx}
\usepackage[colorlinks=true, allcolors=blue]{hyperref}

\title{Bilinear optimal control for chemotaxis model: The case of two-sidedly degenerate  diffusion with Volume-Filling Effect}
\author{ Georges Chamoun$^{2}$, Mazen Saad$^1$, Toni Sayah$^{3}$ and Sarah Serhal$^{1,3}$  }
\date{}
\begin{document}
\maketitle
\begin{center}
$^{1}$ Ecole Centrale de Nantes, 
Laboratoire de Math\'ematiques Jean Leray, UMR CNRS 6629, 1 rue de la No\'e, 44321 Nantes, France. \,  sarah.serhal@ec-nantes.fr\,, mazen.saad@ec-nantes.fr\\

$^{2}$ Faculty of engineering (ESIB), Saint Joseph University  of Beirut, Lebanon.
%des classes préparatoires, Faculté des ingénieurs (ESIB) , Université Saint-Joseph de Beyrouth,  Liban.%
georges.chamoun1@usj.edu.lb\\

$^{3}$ Laboratoire de Mathématiques et Applications, Unité de recherche Mathématiques et
Modélisation, CAR, Faculté des Sciences, Université Saint-Joseph de Beyrouth, B.P 11-514 Riad El Solh, Beyrouth 1107 2050, Lebanon.\\
sarah.serhal@net.usj.edu.lb\,, toni.sayah@usj.edu.lb

\end{center}

\begin{abstract}

\noindent In this paper, we study an optimal control problem for a coupled non-linear system of reaction-diffusion equations with degenerate diffusion, consisting of two partial differential equations representing the density of cells and the concentration of the chemotactic agent. By controlling the concentration of the chemical substrates, this study can guide the optimal growth of cells. The novelty of this work lies on the direct and dual models that remain in a weak setting, which is uncommon in the recent literature for solving optimal control systems. Moreover, it is known that the adjoint problems offer a powerful approach to quantifying the uncertainty associated with model inputs. However, these systems typically lack closed-form solutions, making it challenging to obtain weak solutions. For that, the well-posedness of the direct problem is first well guaranteed. Then, the existence of an optimal control and the first-order optimality conditions are established. Finally, weak solutions for the adjoint system to the non-linear degenerate direct model, are introduced and investigated.
\end{abstract}

\begin{quotation}


\section{Introduction}
Chemotaxis describes the directed movement of organisms in response to chemical gradients. Keller and Segel, introduced in  \cite{keller1970initiation}, a mathematical model that explains the chemotactic aggregation of cellular slime molds. These molds have a tendency to migrate to areas with higher concentrations of chemicals, particularly those secreted by amoebae, a phenomenon known as chemo-attraction with production. In contrast, chemo-repulsion occurs when high chemical concentrations repel organisms.\\

\noindent Bilinear control problems are a special type of control problem in which a non-linear term results from the multiplication of control and state variables, restricted in a control area. It offers a robust and practical method for managing the concentration of substances. Essentially, the control acts as a coefficient of a reaction term that is linearly dependent on the state. In this work, we study an optimal control issue in a system impacted by chemo-attraction that includes a non-linear diffusion factor. To obtain a desired density of cells and concentration of chemicals, this control term affects the injection or extraction of a chemical substance inside a certain subdomain. Specifically, let $\Omega \subset \mathbb{R}^{2}$ be a bounded domain with boundary $\partial \Omega$ of class $\mathcal{C}^{2}$ and $(0, T)$ a time interval, with $0<T<+\infty$. The control problem is given by the following system:
\begin{equation}\label{CNS}
\left\{
\begin{array}{llclr}
&\partial_t N- \nabla \cdot \big(a(N) \nabla N\big) + \nabla \cdot \big(\chi (N) \nabla C\big)  =0 \qquad & \mbox{ in }    (0,T) \times \Omega ,\\[1em]
&\partial _{t} C -\Delta C =\gamma N -\beta C + fC1_{\Omega_c} \qquad & \mbox{ in } (0,T) \times \Omega ,
\end{array}
\right.
\end{equation}
where $f: Q_c:=   (0, T) \times \Omega_c\rightarrow \mathbb{R}$ is the control with $\Omega_c \subset \Omega \subset \mathbb{R}^2$ the control domain, and the unknowns $N, C$ are the cellular density and chemical concentration, respectively. The homogeneous Neumann boundary conditions over $\Sigma_{T}=  (0,T) \times \partial\Omega $ are,
\begin{equation}\label{cl}
a(N)\nabla N \cdot \eta = 0 , \nabla C \cdot \eta = 0, 
\end{equation}
where $\eta$ is the exterior unit normal to $\partial \Omega$. The initial conditions on  $\Omega$ are given by, 
\begin{equation}\label{ci}
N(0,x) = N_{0}(x), C(0,x)= C_{0}(x).
\end{equation}\

\noindent We denote the diffusion coefficient(respectively the convection coefficient) by $a(N)$ (resp. $\chi(N)$). The general form of $\chi(N)$ is $Nh(N)$ where $h(N)$ is the sensitivity function of the chemical. For example, certain bacteria, such as "Bacillus subtilis", have the ability to move toward areas with higher oxygen concentrations in order to survive. On the other hand, in organisms like "Dictyostelium discoideum", chemicals may be produced and consumed, which aids in the transition to becoming multicellular organisms (see \cite{tuval2005bacterial}). This function $\chi$ vanishes in the absence of population $(N=0)$ and above a certain maximum density. This condition has a strong biological significance: once the cell density in a particular location within $\Omega$ surpasses the threshold value of $N=N_{m}$, the cells no longer accumulate there. This phenomenon also referred to as the threshold cell density effect or volume-filling effect, has been a key component in previous studies on chemotaxis in \cite{hillen2001global}.
%==========%
\noindent We denote the control function by $f = f(t,x)$, and this control can be either positive or negative depending on the region of the domain. This means that in areas where 
$f\geq 0$, the substance is being injected into the system, while in areas where 
$f\leq 0$, the substance is being extracted. It is worthwhile noting that using this form of control enables us to ensure the non-negativity of the unknown parameter $C$ in System \eqref{CNS}. However, if we replace $fC$  with just a distributed control $f$ in the second equation of \eqref{CNS}, we would need to assume that $f \geq 0$ to guarantee the positivity of the  concentration $C$.\\
%====================================================================================%
%\subsection{Previous result}
%========================================%
%\subsubsection{System without control}

\noindent In recent decades, there has been a significant increase in research activity focused on the study of the chemotaxis model, which explains the attraction (resp. the repulsion) of living cells towards (resp. away from) the chemical substrates. The system described by \eqref{CNS}, in the absence of control (i.e., $f = 0$), has been explored with a specific emphasis on its non-linear degenerate diffusion term. Numerous studies have also investigated scenarios without the degeneracy of diffusion terms. Notably, the authors in \cite{bendahmane2007two} and \cite{laurenccot2005chemotaxis} %\sarah{ mettre reference bendhmne et laurencot}
established the global existence and uniqueness of smooth classical solutions in $2D$ domains and demonstrated the global existence of weak solutions in dimension $3D$. This is facilitated by the fact that all solutions reside within $H^1(\Omega)$. The persistent challenge of addressing non-linearity and degeneracy in the diffusive term adds a layer of complexity, especially in the context of optimal control in this work. The exploration of degenerate diffusion problems began by Laurencot and Wrzosek in \cite{laurenccot2005chemotaxis} with degeneracy at a single point. They have successfully established their existence while imposing conditions to ensure the uniqueness of solutions. Subsequently, in \cite{bendahmane2007two}, the authors investigated the case of degeneracy at two points and established both the existence and uniqueness of solutions with an additional convective term $\chi(N)\nabla C$. Consequently, numerous studies, have been conducted to explore the modified Keller-Segel model both theoretically and numerically, see for instance \cite{andreianov2011finite,chamoun2014monotone,ibrahim2014efficacy}.\\
%\subsubsection{Controlled system}

\noindent Recently, there has been a surge in interest among researchers towards optimal control problems governed by interconnected partial differential equations. Specifically,  distributed optimal control issues concerning a diffuse interface model of tumor growth have been explored in \cite{colli2017optimal}. A separate discussion by Liu and Yuan in \cite{liu2022optimal} centered on the optimal control of a fully parabolic attraction-repulsion chemotaxis model with a logistic source in a two-dimensional setting. Additionally, in \cite{dai2019optimal}, the authors presented a thorough investigation of an optimal control problem for a haptotaxis model of solid tumor infiltration, considering the use of multiple treatments for cancer. Recently, \cite{guillen2020bi} analyzed the parabolic chemo-repulsion model with a nonlinear production in two-dimensional domains:

\[
\begin{cases}
   \displaystyle u_{t}-\Delta u=\nabla \cdot(u \nabla v), & x \in \Omega, t>0, \\
   \displaystyle v_{t}-\Delta v+v=u^{p}+fv 1_{\Omega_{c}}, & x \in \Omega, t>0 \, .
\end{cases}
\]

\noindent They established the existence and uniqueness of a global-in-time strong state solution for each control, along with the existence of a global optimum solution. In \cite{guillen2020optimal}, the authors considered a bilinear optimal control problem associated with the $2D$ chemo-repulsion model. The existence, uniqueness and regularity of strong solutions for this model were deduced, and for weak solutions in \cite{braz2022bilinear}. They derived the first-order optimality conditions using a Lagrange multipliers theorem.

\noindent 
Unfortunately, there remains a gap in the literature concerning optimal control for chemotaxis problems with degenerate diffusion. Indeed, recent research has explored a bilinear optimal control problem even for the chemotaxis-Navier-Stokes model, but with linear non-degenerate diffusion terms, in bounded three-dimensional domains (see \cite{lopez2021optimal}). In this paper, we face the challenge of studying the solutions' weak sense for the optimal control of chemotaxis models \eqref{CNS}-\eqref{ci} with two-sidedly degenerate diffusive function. This degeneracy not only complicates the existence of weak solutions for the direct problem but also poses challenges in the context of the adjoint problem, despite its linearity. Additionally, there is a persistent difficulty in consistently ensuring the discrete attainment of the maximum principle.\\
%=========================================================================%
%\subsection{Aim of the  paper}
%In this paper, 

\noindent The outline of the paper is as follows:
\begin{itemize}
\item In Section 2, we introduce some notations and assumptions that are usefull for the study of the problem.
\item In Section 3, we establish the existence of weak solutions for the non-degenerate system through a semi-discretization method. We demonstrate that the semi-discretization solution satisfies the maximum principle using specific techniques. Subsequently, we explore the regularization by letting the regularity tend to zero. Moreover, we provide proof of the uniqueness of weak solutions for the degenerate system.
\item In Section 4, we present our second result, focusing on the existence of optimal control. We derive the adjoint problem and explore the first-order optimality conditions. Finally, we establish the existence of a weak solution for the adjoint problem. For that, we define a regularized system to overcome the degeneracy and then demonstrate its existence using the Faedo-Galerkin method. Subsequently,  under further assumptions on the regularity of cell density of the direct problem, we establish the existence of the weak solutions of the adjoint problem by passing to the limit when $\varepsilon$ toward zero.
\item A conclusion is provided in Section 5 to summarize this work. 
\end{itemize}

%%%%%%%%%%%%%%%%%%%========================================================================================================================================================================================================================================================================================================================================%
\section{Preliminaries and Main results}
%\subsection{Problem Setting}

First, we introduce the main notations that we will use later on.
%\begin{lemma}\label{Sobolev}
%Let $\Omega \subset \mathbb{R}^n$ be a bounded domain of class $C^1$. For $a$:
%\begin{itemize}
%    \item If $p<n$, then $W^{1, p}(\Omega) \subset L^q(\Omega)$ for all $q \in \left[1, p^*\right[$, where $\frac{1}{p^*}=\frac{1}{p}-\frac{1}{n}$,
%    \item If $p=n$, then $W^{1, p}(\Omega) \subset L^q(\Omega)$ for all $q \in [1, \infty[$,
%    \item If $p>n$, then $W^{1, p}(\Omega) \subset C(\bar{\Omega})$ with compact embeddings.
%\end{itemize}
%In particular, $W^{1, p}(\Omega) \subset L^p(\Omega)$ with a compact embedding for all $p$.
%\end{lemma}

\noindent We assume that the data satisfies the following properties:
%The main data  assumptions needed are:
\begin{align}
& a: [0,1] \longmapsto \mathbb{R}^{+}, a \in \mathcal{C}^{1}[0,1], a(0)=a(1)=0 \mbox{ and } a(s) > 0 \mbox{ for } 0 < s < 1,\label{aa} \\[1em]
& \chi : [0,1] \longmapsto \mathbb{R},\chi  \in \mathcal{C}^{1}[0,1]\; \mbox{and}\; \chi(0)=\chi(1)=0, \label{chi} \\[1em]
& f\in  L^\infty([0,T]\times\Omega) \, .\label{ff}
\end{align}

\noindent The first main result is to prove the existence and uniqueness of a weak solution of \eqref{CNS}-\eqref{ci} and its continuous dependency 
%of the weak solutions (N,C) 
with respect to the control $f$ .

\noindent We start by giving the definition of a weak solution of system \eqref{CNS}-\eqref{ci}.
\begin{definition} \label{solution faible}
Assume that $0 \leq N_{0} \leq 1$,  $C_{0} \geq 0$ and $C_{0} \in L^{\infty}(\Omega)$. A pair  $(N,C)$ is said to be a {\it{\bf weak solution}} of \eqref{CNS}-\eqref{ci} if
\begin{align*}
&0 \leq N(t,x) \leq 1, \, C(t,x) \geq 0  \mbox{ a.e. in $Q_{T}=  (0,T) \times \Omega $, }\\
&N \in \mathcal{C}_{w}(0,T;L^{2}( \Omega)), \, \quad \partial_{t}N \in  L^{2}(0,T;(H^{1}( \Omega))^{'}), \, \quad  A(N):= \int_0^N a(r) dr \in L^{2}(0,T;H^{1}( \Omega))\, ,\\
&C \in L^{\infty}(Q_{T}) \cap L^{2}(0,T;H^{1}( \Omega)) \cap \mathcal{C}(0,T; L^{2}(\Omega)); \,  \quad \partial_{t}C \in  L^{2}(0,T;(H^{1}( \Omega))')\, ,
\end{align*}
and $(N,C)$ satisfy
\begin{align}
& \int_0^T <\partial_{t}N,\psi_{1}>_{(H^{1})',H^{1}} dt+ \iint_{Q_{T}}{[a(N)\nabla N-\chi(N) \nabla C]  \cdot \nabla \psi_{1}} \; dx dt=   0 ,\label{nhn} \\
& \int_0^T <\partial_{t}C,\psi_{2}>_{(H^{1})',H^{1}} dt+ \iint_{Q_{T}}{\nabla C \cdot \nabla \psi_{2}} \; dx dt= \iint_{Q_{T}}{ [\alpha N -\beta C + {f\hspace{0.0001em}C \hspace{0.0001em} 1_{\Omega_c}}]\psi_{2}} \; dx dt\, ,\label{chc}
\end{align}
for all $\psi_{1}$ and  $\psi_{2}  \in L^{2}(0,T;H^{1}(\Omega))$, where $\mathcal{C}_{w}(0,T;L^{2}(\Omega))$ denotes the space of continuous functions with values in  $L^{2}(\Omega)$ endowed with the weak topology,
and $<\cdot,\cdot>$ denotes the duality pairing between $H^1(\Omega)$ and $(H^1(\Omega))'$.
\end{definition}
%======================================================% 

\noindent We now present our primary results concerning the global existence and uniqueness of weak solutions in the following theorems (proved in the next section):
\begin{theoreme}\label{th:existence}
Assume that \eqref{aa} to \eqref{ff} hold. If \, $0 \leq N_{0} \leq 1$, $C_{0} \geq 0$ a.e. in $\Omega$ and  $C_{0} \in L^{\infty}(\Omega)$, then System \eqref{CNS}-\eqref{cl}-\eqref{ci} has a global weak solution $(N,C)$ in the sense of Definition \ref{solution faible}.
\end{theoreme}
%===========================================%
\begin{theoreme}\label{th:unicite}
The weak solution (N,C) of \eqref{CNS}-\eqref{ci} is unique under the following assumption: there exists $K_{0}\geq 0$ such that,
\begin{align}\label{unicite}
(\chi(N_{1})-\chi(N_{2}))^{2} \leq K_{0} (N_{1}-N_{2}) (A(N_{1})-A(N_{2})), \, \, \, \forall N_{1},N_{2} \in [0,1]\, .
\end{align}
\end{theoreme}
%
%===========================%

%\noindent The second main result of this paper will be the existence of a global optimal solution  of the optimal problem defined later in section 4.\\

%\noindent Finally, we will prove  the existence of Lagrange multipliers associated to the optimal control under further strong regularity on N.

%=================================================================================================%
%=====================================================%
%==================================%
\section{Global existence and uniqueness of solutions of problem \eqref{CNS}-\eqref{ci}}

The aim of this section is to demonstrate briefly  Theorems \ref{th:existence} and \ref{th:unicite}, since their proofs follow the same guidelines of \cite{chamoun2014coupled}.
Compared to \cite{chamoun2014coupled}, our approach emphasizes positivity and boundedness irrespective of the control's sign. For that, in our numerical scheme (explained in the next subsection), we have incorporated two discretizations of $f$, one positive and one negative, to guarantee first the existence of discrete solutions for the non-degenerate problem while ensuring that the chemoattractant's positivity and boundedness are preserved, independent of the regularization.
Note that the main difficulty is the strong degeneracy of the diffusion term and the conservation of the maximum principle while introducing a control term to the system. To navigate the first obstacle, we replace the original diffusion term $a(N)$ by $a_{\varepsilon}(N) = a(N) + \varepsilon$ and we consider for each fixed $\varepsilon > 0$, the following non-degenerate problem:
\begin{equation}
\label{CNS1}
\left\{
\begin{array}{rclr}
&\partial _{t} N_{\varepsilon} -\nabla \cdot(a_{\varepsilon}(N_{\varepsilon}) \nabla N_{\varepsilon}) + \nabla \cdot (\chi (N_{\varepsilon}) \nabla C_{\varepsilon})= 0,\\[0.5em]
&\hspace{-5em}\partial _{t} C_{\varepsilon} -\Delta C_{\varepsilon} = \alpha N_{\varepsilon}-\beta C_{\varepsilon} + {f\hspace{0.0001em}C\hspace{0.0001em} 1_{\Omega_c}} ,
\end{array}
\right.
\end{equation}
with the following  conditions,
\begin{equation}\label{bord}
\left\{
\begin{array}{rclr}
&a(N_{\varepsilon})\nabla N_{\varepsilon}\cdot \eta = 0, \, \nabla C_{\varepsilon}  \cdot \eta=0 \mbox { on }  (0,T) \times \partial \Omega \, ,\\[0.5em]
&\hspace{-5em} N_{\varepsilon}(0,x) = N_{0}(x), \, C_{\varepsilon}(0,x)= C_{0}(x) .
\end{array}
\right. 
\end{equation}
Using a semi-discretization in time, the existence of weak solutions to the non-degenerate problem \eqref{CNS1} is well guaranteed in the following subsection. Then, we ensure that solutions satisfy the maximum principle, the suitable estimates and compactness arguments. Next, the sequence of approximate solutions $(N_{\varepsilon}, C_{\varepsilon})$ converges to weak solutions $(N,C)$ of \eqref{CNS}-\eqref{ci} in the sense of definition \ref{solution faible}, when $\varepsilon$ tends to zero.
%==================================================%

\subsection{Construction of an approximating solution}\label{pr}
In this subsection, we show the existence of a weak solution to system \eqref{CNS1}-\eqref{bord} by following the method introduced in \cite{marpeau2006mathematical} for systems modeling the miscible displacement of radioactive elements in a heterogeneous porous domain, and in \cite{chamoun2014coupled} for a coupled anisotropic chemotaxis-fluid model. 
A semi-discretization in time is used to prove the existence result for any data $f$. For the reader's convenience, we detail the main steps of our method. 

\noindent We decompose $f$ as the sum of its positive and negative components denoted respectively as $f^{+}$ and $f^{-}$,  and  defined as follows: $f = f^+ + f^{-}$ such that
\[ f^{+} = \max(f, 0) \quad \mbox{ and } \quad  f^{-} = \max(-f, 0). \]
%The original function $f$ can then be written as the sum of these components: $f = f^{+} - f^{-}$.

%\noindent Over a Banach space $\mathcal{E}$, we define two interpolation operators
 
\noindent Consider constant time step  $h=\Delta t = \frac{T}{\tilde{N}}$, $T>0$ and  $\tilde{N}\in \mathbb{N}^{*}$.
   
\noindent Let $\mathcal{E}$ be a given Banach space. For all $ v=\left(v_{0}, v_{1}, \ldots, v_{\tilde{N}}\right) \in \mathcal{E}^{N+1}$, we introduce following operators:
\begin{enumerate}
\item  The constant interpolation operator defined from $[0,T]$ to $\mathcal{E}$ given by
\begin{equation*}
\displaystyle\left\{
\begin{array}{clcc}
\Pi_{\tilde{N}}^{0} v(0)= v_{0}\, ,\\
\displaystyle \Pi_{\tilde{N}}^{0} v(t)= \sum_{n=0}^{\tilde{N}-1} v_{n+1} \chi_{]nh,(n+1)h]}(t)\;, \mbox{ if } 0<t \leq T\, ,
\end{array}
\right.
\end{equation*}
where $\chi_{]nh,(n+1)h]}(t)$ being the characteristic function in $ ]nh,(n+1)h]$.
\item The linear interpolation operator from $[0,T]$ to $\mathcal{E}$ defined by
\begin{align*}
&\Pi_{\tilde{N}}^{1} v(t)= \sum_{n=0}^{\tilde{N}-1} \Big[(1+n-\frac{t}{h})v_{n}+(\frac{t}{h}-n)v_{n+1}\Big] \chi_{[nh,(n+1)h]}(t)\,,\\
&\frac{d}{dt} \Big(\Pi_{\tilde{N}}^{1} v(t)\Big)= \sum_{n=0}^{\tilde{N}-1} \Big[\frac{1}{h}(v_{n+1}-v_{n})\Big] \chi_{]nh,(n+1)h]}(t)\, .
\end{align*}  
\end{enumerate}
\noindent Remark that
\begin{align}\label{defpi}
\left\|\Pi_{\Bar{N}}^{0} v\right\|_{L^{p}((0, T) ; \mathcal{E})}=\left(h \sum_{n=1}^{\tilde{N}}\left\|v_{n}\right\|_{\mathcal{E}}^{p}\right)^{1 / p},\quad \text { if } 1 \leqslant p<\infty,
\end{align}
\begin{align}\label{infini}
\left\|\Pi_{\tilde{N}}^{0} v\right\|_{L^{\infty}((0, T) ; \mathcal{E})}=\left\|\Pi_{\tilde{N}}^{1} v\right\|_{L^{\infty}((0, T) ; \mathcal{E})}=\max _{n=1, \ldots, \Bar{N}}\left(\left\|v_{n}\right\|_{\mathcal{E}}\right)  .
\end{align}
%=================================================%%
Next, for all functions   $\zeta$ in $L^{1}((0, T) ; \mathcal{E})$, we define the averaging operator   $\Lambda^{\tilde{N}}$ by\\ $\Lambda^{\tilde{N}} \zeta^{\tilde{N}}=$ $((\Lambda^{\tilde{N}} \zeta)_{1}, \ldots,(\Lambda^{\tilde{N}} \zeta)_{\tilde{N}+1}) \in \mathcal{E}^{\tilde{N}+1}$, with 
\begin{align}
%&\left(\Lambda^{\Bar{N}} \zeta\right)_{0}=0,\;\; \mbox{and}
& \left(\Lambda^{\Bar{N}} \zeta\right)_{n+1}=\left(\Lambda^{\tilde{N}} \zeta\right)^{+}_{n+1}-\left(\Lambda^{\tilde{N}} \zeta\right)^{-}_{n+1} =\frac{1}{h} \int_{n h}^{(n+1) h} \zeta^{+}(t) d t-\frac{1}{h} \int_{n h}^{(n+1) h} \zeta^{-}(t) d t, \quad \text { for } 0\leqslant n \leqslant \tilde{N}.
\end{align}
If  $\zeta \in L^{p}((0, T) ; \mathcal{E})$, there holds 
\begin{equation}\label{aprox chakl f}
 \left\|\Pi_{\tilde{N}}^{0} \Lambda^{\tilde{N}} \zeta\right\|_{L^{p}((0, T) ; \mathcal{E})}\leq  \left\| \zeta\right\|_{L^{p}((0, T) ; \mathcal{E})},\quad \text{if} \ 1 \leq p \leq \infty,
 \end{equation}
\begin{equation}\label{cv fort chakel f}
   \begin{split}
     \quad\quad\quad\quad \quad&\Pi_{N}^{0} (\Lambda^{\tilde{N}} \zeta)^{+} \underset{\tilde{N} \rightarrow       \infty}{\longrightarrow}{\zeta^{+}}  \quad \text{strongly in } L^{p}((0, T)   ; \mathcal{E}), \quad \text{if} \  1\leq p \leq \infty,\\
 & \Pi_{N}^{0} (\Lambda^{\tilde{N}} \zeta)^{-} \underset{\tilde{N} \rightarrow \infty}{\longrightarrow}{\zeta^{-}}  \quad \text{strongly in } L^{p}((0, T) ; \mathcal{E}), \quad \text{if} \  1\leq p \leq \infty \, .
\end{split}
\end{equation}
%==============================%
Next, we define a family of approximate solutions by the following  discretized scheme (in time):  %to construct a sequence of approximate solutions.

\noindent Setting $f^{\tilde{N}}=\Lambda^{\tilde{N}} f$, we define $N^{\tilde{N}}_{\varepsilon}=(N_{0,\varepsilon}^{\tilde{N}}, N_{1,\varepsilon}^{\tilde{N}}, ...,N_{\tilde{N},\varepsilon}^{\tilde{N}})$ and $C^{\tilde{N}}=( C_{0,\varepsilon}^{\tilde{N}}, C_{1,\varepsilon}^{\tilde{N}},...,C_{\tilde{N},\varepsilon}^{\tilde{N}})$ such that:  
%$  N_0^{\tilde{N}}=N_0$ and $  C_0^{\tilde{N}}=C_0$. Now, we begin with
%\begin{align}
$$(N_{0,\varepsilon}^{\tilde{N}},C_{0,\varepsilon}^{\tilde{N}}) = (N_{0},C_{0}) \mbox{ the given initial data,}$$
%\end{align} 
for each $0\le n \le \tilde{N}-1$, having  $(N_{n,\varepsilon}^{\tilde{N}},C_{n,\varepsilon}^{\tilde{N}})$,
%...,(N_{n,\varepsilon}^{\tilde{N}},C_{n,\varepsilon}^{\tilde{N}})$ are known, 
we compute $(N_{n+1,\varepsilon}^{\tilde{N}},C_{n+1,\varepsilon}^{\tilde{N}})$  solution of the following system: $\forall\; \psi_{1}, \psi_{2} \in H^{1}(\Omega) $,
%===============%
\begin{equation}\label{eqn}
\hspace{-6em}\displaystyle\frac{1}{h} \int_{\Omega}{(N_{n+1,\varepsilon}^{\tilde{N}}-N_{n,\varepsilon}^{\tilde{N}})\psi_{1}} \; dx +\int_{\Omega}{\Big(a_{\varepsilon}(N_{n+1,\varepsilon}^{\tilde{N}}) \nabla N_{n+1,\varepsilon}^{\tilde{N}}-\chi (N_{n+1,\varepsilon}^{\tilde{N}}) \nabla C_{n+1,\varepsilon}^{\tilde{N}}\Big) \cdot \nabla \psi_{1}}\; dx =0 \, ,
 \end{equation}

\begin{equation}\label{eqc}
\begin{split}
\displaystyle\frac{1}{h} \int_{\Omega}{(C_{n+1,\varepsilon}^{\tilde{N}}-C_{n,\varepsilon}^{\tilde{N}})\psi_{2}} \; dx  + \int_{\Omega}{\nabla C_{n+1,\varepsilon}^{\tilde{N}} \cdot \nabla \psi_{2}}\; dx=& \displaystyle\int_{\Omega}\Big(\alpha{{N_{n,\varepsilon}^{\tilde{N}}}-\beta C_{n+1,\varepsilon}^{\tilde{N}}\Big)\psi_{2}}\; dx\\
&+\displaystyle\int_{\Omega_c}\big(\; {(f_{n+1}^{\tilde{N}})^{+}C_{n,\varepsilon}^{\tilde{N}}-(f_{n+1}^{\tilde{N}})^{-}C_{n+1,\varepsilon}^{\tilde{N}}\big)
\; \psi_{2}}\; dx \, .
\end{split}
\end{equation}
Note that the approximation of the bilinear control is given explicitly for the positive part of the control and implicitly for the negative part of the control $f$ which ensures the positivity of the chemo-attractant concentration. On the other hand, the production term in the concentration equation is given explicitly which allows us to decouple the system.
%============================================================================================================================================%
\subsection{Confinement of solutions} 
A biological admissible solution is bounded, 
 that's why the discrete solution must be also bounded. The last property will be shown in the following by using the discrete maximum principle of \eqref{eqn}-\eqref{eqc} .%which shows the importance of the discrete maximum principle for weak solutions of \eqref{CNS1}.
\begin{proposition}\label{maxprinciple}
Each approximate solution is biologically admissible, i.e
$$\exists M\geqslant 0, \forall n=0,...,\tilde{N}-1, 
 0 \leq N_{n+1,\varepsilon}^{\tilde{N}} \leq 1 \mbox{ and } 0 \leq C_{n+1,\varepsilon}^{\tilde{N}} \leq M \mbox{ a.e. } x\in \Omega.$$ 
\end{proposition}
%==========================================%
\begin{proof}
The proof for $N_{n+1}^{\tilde{N}}$ is given by \cite{chamoun2014coupled}. We will treat only the %novelty control term introduced in% 
concentration equation of our model which contains the control term. Given $C^{-}=\max (-C, 0)$ in $H^{1}(\Omega)$, we prove now that  $C_{n+1, \varepsilon}^{\tilde{N}}\geq 0$. One has $C_{0, \varepsilon}^{\tilde{N}}=C_{0} \geq 0$. Suppose that  $C_{n, \varepsilon}^{\tilde{N}} \geq 0$ and  choose $\psi_{2}=-\left(C_{n+1, \varepsilon}^{\tilde{N}}\right)^{-}$ as a test function in \eqref{eqc}, then
 \begin{equation*}
 \begin{split}
  \displaystyle\frac{1}{h} \int_{\Omega}|C_{n+1,\varepsilon}^{\tilde{N}})^{-}|^2\; dx =&-\frac{1}{h}\int_{\Omega}{ C_{n,\varepsilon}^{\Bar{N}}(C_{n+1,\varepsilon}^{\tilde{N}})^{-}}\; dx + \int_{\Omega}{\nabla (C_{n+1,\varepsilon}^{\tilde{N}} )\cdot \nabla ( C_{n+1,\varepsilon}^{\tilde{N}})^-}\; dx
  \\
 & -\int_{\Omega}{\alpha N_{n}^{\tilde{N}}(C_{n+1,\varepsilon}^{\tilde{N}})^{-}}\; dx 
     + \int_{\Omega}{\beta\ C_{n+1,\varepsilon}^{\tilde{N}}(C_{n+1,\varepsilon}^{\tilde{N}})^{-}}\; dx\\
   & -\int_{\Omega_c}{({f_{n+1}^{\tilde{N}})^{+}}{C_{n,\varepsilon}^{\tilde{N}}}(C_{n+1,\varepsilon}^{\tilde{N}})^{-}}\; dx
    +\int_{\Omega_c}{({f_{n+1}^{\tilde{N}})^{-}}{C_{n+1,\varepsilon}^{\tilde{N}}}(C_{n+1,\varepsilon}^{\tilde{N}})^{-}}\; dx.\\
    \leqslant &-\frac{1}{h}\int_{\Omega}{ C_{n,\varepsilon}^{\Bar{N}}(C_{n+1,\varepsilon}^{\tilde{N}})^{-}}\; dx -\int_{\Omega}\abs{\nabla (C_{n+1,\varepsilon}^{\tilde{N}} )^-}^2\; dx -\int_{\Omega}{\alpha N_{n}^{\tilde{N}}(C_{n+1,\varepsilon}^{\tilde{N}})^{-}}\; dx \\
      &-\int_{\Omega}\beta\abs{(C_{n+1,\varepsilon}^{\tilde{N}})^{-}}^2\; dx -\int_{\Omega_c}{({f_{n+1}^{\tilde{N}})^{+}}{C_{n,\varepsilon}^{\tilde{N}}}(C_{n+1,\varepsilon}^{\tilde{N}})^{-}}\; dx \\
      &-\int_{\Omega_c}({f_{n+1}^{\tilde{N}})^{-}}\abs{(C_{n+1,\varepsilon}^{\tilde{N}})^{-}}^2\; dx.
    \end{split}
\end{equation*}
\noindent Since $N_{n,\varepsilon}^{\tilde{N}},\ \alpha,\ \beta$ and $(f_{n+1}^{\tilde{N}})^{\pm}$ are positive and according to  the negativity of the  terms  of the right-hand-side, we obtain
\begin{align*}
\frac{1}{h} \int_{\Omega}\abs{(C_{n+1,\varepsilon}^{\tilde{N}})^{-}}^2\; dx \leq 0\, .
 \end{align*}
Therefore, $(C_{n+1,\varepsilon}^{\tilde{N}})^{-}=\max(-C_{n+1,\varepsilon}^{\tilde{N}}, 0)=0$,
%which is in contradiction with $C_{n+1,\varepsilon}^{\tilde{N}} < 0$. 
consequently, 
\begin{align}\label{geo3}
\forall n=0,..,\tilde{N}-1, \; C_{n+1,\varepsilon}^{\tilde{N}} \geq 0 \mbox{ a.e. }  x \in \Omega \, .
\end{align}
%================================%
%===========================%

\noindent Let us now focus on the %last claim concerning%
the existence of a constant $M$ such that $C_{n+1,\varepsilon}^{\tilde{N}}\leq M$. 
We assume that $C_{n,\varepsilon}^{\tilde{N}} \leq M. $ Multiplying \eqref{eqc} by $h$ and  choose $(C_{n+1,\varepsilon}^{\tilde{N}}-M\big)^{+}$ as the test function, lead to
\begin{equation*}
\begin{split}
&\displaystyle{  \int_{\Omega}\left(C_{n+1,\varepsilon}^{\tilde{N}}-C_{n,\varepsilon}^{\tilde{N}}\right)\left(C_{n+1,\varepsilon}^{\tilde{N}}-M\right)^{+}\; dx }\\
&\hspace{3em}=-h\int_{\Omega}\abs{\nabla(C_{n+1,\varepsilon}^{\tilde{N}}-M)^{+}}^{2}\; dx{+ h\; \alpha \int_{\Omega}N_{n,\varepsilon}^{\tilde{N}}\left(C_{n+1,\varepsilon}^{\tilde{N}}-M\right)^{+}\; dx }
\\&\hspace{4em}
-h\;\beta \int_{\Omega}C_{n+1,\varepsilon}^{\tilde{N}}\left(C_{n+1,\varepsilon}^{\tilde{N}}-M\right)^{+}\; dx\displaystyle\quad+h\int_{\Omega_c}(f_{n+1}^{\tilde{N}})^{+}C_{n,\varepsilon}^{\tilde{N}}\left(C_{n+1,\varepsilon}^{\tilde{N}}-M\right)^{+}\; dx 
\\
&\hspace{4em}\displaystyle -h\int_{\Omega_c}(f_{n+1}^{\tilde{N}})^{-}C_{n+1,\varepsilon}^{\tilde{N}}\left(C_{n+1,\varepsilon}^{\tilde{N}}-M\right)^{+}\; dx \, %====================%
\\&\hspace{3em}\leqslant { h\; \alpha \int_{\Omega}N_{n,\varepsilon}^{\tilde{N}}\left(C_{n+1,\varepsilon}^{\tilde{N}}-M\right)^{+}\; dx }\displaystyle+h\int_{\Omega_c}(f_{n+1}^{\tilde{N}})^{+}C_{n,\varepsilon}^{\tilde{N}}\left(C_{n+1,\varepsilon}^{\tilde{N}}-M\right)^{+}\; dx,
\end{split}
\end{equation*}%==========================================%
Therefore, using that $N_{n,\varepsilon}^{\tilde{N}}\in [0,1] $, we get  
 \begin{equation*}%=================================%
\displaystyle{  \int_{\Omega}\left(C_{n+1,\varepsilon}^{\tilde{N}}-C_{n,\varepsilon}^{\tilde{N}}-h (f_{n+1}^{\tilde{N}})^{+}C_{n,\varepsilon}^{\tilde{N}}-h \alpha\right)\left(C_{n+1,\varepsilon}^{\tilde{N}}-M\right)^{+}\; dx }\leqslant 0.
\end{equation*}%========================%
Now, we choose $M_{n}=\abs{C_{n,\varepsilon}^{\tilde{N}}}(1+h\abs{f_{n+1}^{\tilde{N}}})+h \alpha$. Thus, we have%=====================================%
$$\displaystyle{  \int_{\Omega}\left(C_{n+1,\varepsilon}^{\tilde{N}}-M_n\right)\left(C_{n+1,\varepsilon}^{\tilde{N}}-M\right)^{+}\; dx }\leqslant 0.$$
Our goal is to find $M$ independent of n, such that $M \geqslant M_{n}$ and %==============% 
$$\displaystyle{  \int_{\Omega}\left(C_{n+1,\varepsilon}^{\tilde{N}}-M\right)\left(C_{n+1,\varepsilon}^{\tilde{N}}-M\right)^{+}\; dx }\leqslant\displaystyle{  \int_{\Omega}\left(C_{n+1,\varepsilon}^{\tilde{N}}-M_n\right)\left(C_{n+1,\varepsilon}^{\tilde{N}}-M\right)^{+}\; dx }\leqslant 0.$$%=====================%
%One can%
Let us now show that $C_{n+1,\varepsilon}^{\tilde{N}}\leqslant M, \mbox{for all} \ n\in\{0,...,\tilde{N}-1\}.$ %========================================%

\noindent By using  $1+x \leqslant e^{x}$, we have
$$
\begin{aligned}
 %0 \leqslant %
 C_{n+1,\varepsilon}^{\tilde{N}}
 &\leqslant \left|C_{n, \varepsilon}^{\tilde{N}}\right|\left(1+h|f_{n+1}^{\tilde{N}}|\right)+\alpha h \displaystyle \\&\leqslant\left|C_{n, \varepsilon}^{\tilde{N}}\right|\displaystyle e^{ h|f_{n+1}^{\tilde{N}}|}+\alpha h  \\&\leqslant\left|C_{n, \varepsilon}^{\tilde{N}}\right| e^{\int_{t_{n}}^{t_{n+1}}{|f_{n+1}^{\tilde{N}}|}\; dt}+\alpha h.
 \end{aligned}
$$
Then we get
\begin{align*}
%\\%=============%
C_{n+1,\varepsilon}^{\tilde{N}}%==============%
& \leqslant\left(|C_{n-1, \varepsilon}^{\tilde{N}}| e^{\int_{t_{n-1}}^{t_{n}}{|f_{n}^{\tilde{N}}|}\; dt}+\alpha h\right)\;\left(e^{\int_{t_{n}}^{t_{n+1}}{|f_{n+1}^{\tilde{N}}|}\; dt }\right)+\alpha h\\
%+========================%
&  \leqslant|C_{n-1, \varepsilon}^{\tilde{N}}| e^{\int_{t_{n-1}}^{t_{n+1}}{\norm{f}_{{\infty}}\; dt}}+\alpha h\left(1+ e^{\int_{t_{n}}^{t_{n+1}}{\norm{f}_{\infty}}\; dt}\right)  \\
&  \leqslant|C_{n-2, \varepsilon}^{\tilde{N}}| e^{\int_{t_{n-2}}^{t_{n+1}}{\norm{f}_{{\infty}}\; dt}}+\alpha h\left(1+ e^{\int_{t_{n-1}}^{t_{n}}{\norm{f}_{\infty}}\;}+e^{\int_{t_{n}}^{t_{n+1}}{\norm{f}_{\infty}}\; dt}\right) .
\end{align*}
By successive steps, we deduce that
\begin{equation*}
\begin{split}
M_{n}&\leqslant |C_{0, \varepsilon}^{\tilde{N}}| e^{\int_{0}^{t_{n+1}}{\norm{f}_{{\infty}}\; dt}}+\alpha hn\;e^{\int_{0}^{t_{n+1}}{\norm{f}_{\infty}}\; dt}+\alpha h\\
&\leqslant \norm{C_{0}}_{{\infty}} e^{\int_{0}^{T}{\norm{f}_{{\infty}}\; dt}}+\alpha T\left(1+e^{\int_{0}^{T}{\norm{f}_{\infty}}\; dt}\right).
\end{split}
\end{equation*}
Therefore, we choose $M=\norm{C_{0}}_{{\infty}} e^{\int_{0}^{T}{\norm{f}_{{\infty}}\; dt}}+\alpha T\left(1+e^{\int_{0}^{T}{\norm{f}_{\infty}}\; dt}\right)$.

\noindent This ends the proof of the proposition.
\end{proof}
%=================================================%
%============================================================%
%===================================================================%

\begin{remarque}
In the case of  distributed control $f$ instead of $fC$ in \eqref{CNS}, it is necessary to choose $f \geqslant 0$ to ensure the positivity of the discrete solution $C_{n+1,\varepsilon}^{\tilde{N}}$.

\noindent To guarantee the boundedness of the discrete concentration in the context of the maximum principle (specifically there exists $\rho > 0$ such that $C_{n+1, \varepsilon}^{\tilde{N}} \leqslant \rho$), the following approach is utilized:

\noindent we set $\rho_{n}={\norm{C_0}}_{\infty}+n\left(\alpha+\norm{f}_{L^{\infty(Q_T)}}\right)\Delta t$
and suppose that $C_{n,\varepsilon}^{\tilde{N}} \leq \rho_n$. We will  show  that $C_{n+1,\varepsilon}^{\tilde{N}} \leq \rho_{n+1}$. Suppose that  $C_{n, \varepsilon}^{\tilde{N}} \leq \rho_n$ and choose $(C_{n+1,\varepsilon}^{\tilde{N}}-\rho_{n+1}\big)^{+}$ as the test function in \eqref{eqc}, then

\begin{equation*}
\hspace{-7.5em}\quad\displaystyle{ \frac{1}{h} \int_{\Omega}\left(C_{n+1,\varepsilon}^{\tilde{N}}-C_{n,\varepsilon}^{\tilde{N}}+\rho_{n+1}-\rho_{n+1}\right)\left(C_{n+1,\varepsilon}^{\tilde{N}}-\rho_{n+1}\right)^{+}dx }=-\int_{\Omega}\abs{\nabla(C_{n+1,\varepsilon}^{\tilde{N}}-\rho_{n+1}\big)^{+}}^{2}\; dx
\end{equation*}
$$\hspace{13.7em}{+\alpha \int_{\Omega}N_{n,\varepsilon}^{\tilde{N}}\left(C_{n+1,\varepsilon}^{\tilde{N}}-\rho_{n+1}\right)^{+} dx }-\beta \int_{\Omega}C_{n+1,\varepsilon}^{\tilde{N}}\left(C_{n+1,\varepsilon}^{\tilde{N}}-\rho_{n+1}\right)^{+}\; dx $$
$$\hspace{-3em} \displaystyle+\int_{\Omega_c}f_{n+1}^{\tilde{N}}\left(C_{n+1,\varepsilon}^{\tilde{N}}-\rho_{n+1}\right)^{+}\; dx \, .$$
Therefore, we get
\begin{equation*}
\begin{array}{ll}
\medskip
\displaystyle{ \frac{1}{h} \int_{\Omega}\left(C_{n+1,\varepsilon}^{\tilde{N}}-\rho_{n+1}\right)\left(C_{n+1,\varepsilon}^{\tilde{N}}-\rho_{n+1}\right)^{+}dx} = \displaystyle \frac{1}{h} \int_{\Omega}\left(C_{n,\varepsilon}^{\tilde{N}}-\rho_{n+1}\right)\left(C_{n+1,\varepsilon}^{\tilde{N}}-\rho_{n+1}\right)^{+}\; dx\\
\hspace{4cm} \displaystyle -\int_{\Omega}\abs{\nabla(C_{n+1,\varepsilon}^{\tilde{N}}-\rho_{n+1})^{+}}^{2}dx{+\alpha \int_{\Omega}N_{n,\varepsilon}^{\tilde{N}}\left(C_{n+1,\varepsilon}^{\tilde{N}}-\rho_{n+1}\right)^{+} dx } \\
\hspace{4cm} \displaystyle  -\beta \int_{\Omega}C_{n+1,\varepsilon}^{\tilde{N}}\left(C_{n+1,\varepsilon}^{\tilde{N}}-\rho_{n+1}\right)^{+}dx+\int_{\Omega_c}f_{n+1}^{\tilde{N}}\left(C_{n+1,\varepsilon}^{\tilde{N}}-\rho_{n+1}\right)^{+}\; dx \, .
\end{array}
\end{equation*}
Using now that $\rho_{n+1}=\rho_n +\left(\alpha+\norm{f}_{L^{\infty}(Q_T)}\right)$, the last equation gives
\begin{equation*}
{ \frac{1}{h} \int_{\Omega}\left(\left(C_{n+1,\varepsilon}^{\tilde{N}}-\rho_{n+1}\right)^{+}\right)^{2} dx\leqslant+ \frac{1}{h} \int_{\Omega}\left(C_{n,\varepsilon}^{\tilde{N}}-\rho_{n}\right)\left(C_{n+1,\varepsilon}^{\tilde{N}}-\rho_{n+1}\right)^{+}\; dx} 
\end{equation*}
$$\hspace{7.7em}+\alpha \int_{\Omega}\left(N_{n,\varepsilon}^{\tilde{N}}-1\right)\left(C_{n+1,\varepsilon}^{\tilde{N}}-\rho_{n+1}\right)^{+} dx+\int_{\Omega}\left(f_{n+1}^{\tilde{N}}-\norm{f}_{L^\infty(Q_T)}\right)\left(C_{n+1,\varepsilon}^{\tilde{N}}-\rho_{n+1}\right)^{+}\; dx $$ 
$$\hspace{-15em}-\beta \int_{\Omega}C_{n+1,\varepsilon}^{\tilde{N}}\left(C_{n+1,\varepsilon}^{\tilde{N}}-\rho_{n+1}\right)^{+}\; dx  .$$
Since $\alpha$ and $\beta $  are positives,\  $N_{n+1,\varepsilon}^{\tilde{N}} \leq 1  \mbox{ a.e. }  x \in \Omega \,  , \forall n=0,..,\tilde{N}-1$, and $f_{n+1}^{\tilde{N}}\leq \norm{f}_{L^{\infty}(Q_T)}$, and according to the non-negativity of the second term of the left-hand-side, one can deduce that $C_{n+1}^{\tilde{N}}\leq \rho_{n+1}\leq \rho, \mbox{for all} \ n\in\{0,...,N\}.$ We take then $\rho= {\norm{C_0}}_{\infty}+T\left(\alpha+\norm{f}_{L^{\infty(Q_T)}}\right)$ and get $C_{n+1, \varepsilon}^{\tilde{N}} \leqslant \rho$.
\end{remarque}
%==============================================================%
%=================================================================================%
\subsection{Existence of a weak solution of System \eqref{eqn}-\eqref{eqc} }
Since the term {$(f_{n+1}^{\tilde{N}})^{-}C_{n+1,\varepsilon}^{\tilde{N}}$} is non-negative, the bilinear form in the left-hand side of Equation \eqref{eqc} is coercive.
%of the does not lead to a problem in demonstrating the coercivity of
The Lax-Milgram theorem gives %establish both 
the existence and uniqueness of the discrete solution $C_{n+1,\varepsilon}^{\tilde{N}}$ to Equation \eqref{eqc}.
%, which arises from the semi-discretization in time of equation \eqref{eqc}. Furthermore, we can prove the existence of a unique $C_{n+1,\varepsilon}^{\tilde{N}}$ satisfying \eqref{eqc} by utilizing the Lax-Milgram theorem.
%==================================================================================%
%=====================================================================================%
%\subsection{Existence of a weak solution of the equation \eqref{eqn}}
The Schauder fixed point theorem is the main tool to prove the existence of the $N_{n+1,\varepsilon}^{\tilde{N}}$ of Equation \eqref{eqn}.
This method is used in \cite{chamoun2014coupled} and it is easy to adopt the proof in our case, thus we omit it.\\
%=========================================================================================%
%\subsection{{\it A priori} estimates}

\noindent The following proposition presents a crucial result, where a priori estimates on the interpolation of $N^{\tilde{N}}$ and $C^{\tilde{N}}$ with respect to $\tilde{N}$ are obtained.
\begin{proposition}
There exist positive constants $A'$ and $A''$ independent of $\tilde{N}$ such that
\begin{align}
& \big|\big|\Pi_{\tilde{N}}^{0}  N_{\varepsilon}^{\tilde{N}}\big|\big|_{L^{\infty}(Q_{T})} = \max_{n=1...N} \big|\big|N_{n,\varepsilon}^{\tilde{N}}\big|\big|_{L^{\infty}(\Omega)} \leq 1, \qquad \big|\big|\Pi_{\tilde{N}}^{0} C_{\varepsilon}^{\tilde{N}}\big|\big|_{L^{\infty}(Q_{T})}  \leq M, \label{a}\\
%\end{align}
%\begin{align}
&\big|\big|\Pi_{\tilde{N}}^{0} N_{\varepsilon}^{\tilde{N}}\big|\big|_{L^{2}(0,T;(H^{1}(\Omega)))} \leq A', \qquad \big|\big|\Pi_{\tilde{N}}^{0} C_{\varepsilon}^{\tilde{N}}\big|\big|_{L^{2}(0,T;(H^{1}(\Omega)))} \leq A'' \, ,\label{b,c}\\
%\end{align}
%\begin{align}
&||\tau_{h'}\Pi_{\tilde{N}}^{0} N_{\varepsilon}^{\tilde{N}}-\Pi_{\tilde{N}}^{0} N_{\varepsilon}^{\tilde{N}}||_{L^{2}(0,T-h';(L^{2}(\Omega)))} \leq A'h'\, , \;\;
 \forall h'> 0, \label{V3}\\
%\end{align}
%\begin{align}
 &||\tau_{h'}\Pi_{\tilde{N}}^{0} C_{\varepsilon}^{\tilde{N}}-\Pi_{\tilde{N}}^{0} C_{\varepsilon}^{\tilde{N}}||_{L^{2}(0,T-h';(L^{2}(\Omega)))} \leq A'h' \, ,\;\;  \forall h'> 0, \label{V4}\\
%\end{align}
%\begin{align}
&\Big|\Big|\frac{\partial}{\partial t} (\Pi_{\tilde{N}}^{1} N_{\varepsilon}^{\tilde{N}})\Big|\Big|_{L^{2}(0,T;(H^{1}(\Omega))')} \leq A', \qquad  \Big|\Big|\frac{\partial}{\partial t} (\Pi_{\tilde{N}}^{1} C_{\varepsilon}^{\tilde{N}})\Big|\Big|_{L^{2}(0,T;(H^{1}(\Omega))')} \leq A''\, ,\label{but} \\
%\end{align}
%\begin{align}
& \big|\big|\Pi_{\tilde{N}}^{1} N_{\varepsilon}^{\tilde{N}}-\Pi_{\tilde{N}}^{0} N_{\varepsilon}^{\tilde{N}}\big|\big|_{L^{2}(0,T;(H^{1}(\Omega))')} \leq \frac{A'h^{2}}{3} \, .\label{L3}
\end{align}
\end{proposition}

\noindent Consequently, there exist sub-sequences of $(\Pi_{\tilde{N}}^{0} C_{\varepsilon}^{\tilde{N}})_{\tilde{N}}$ and $(\Pi_{\tilde{N}}^{0} N_{\varepsilon}^{\tilde{N}})_{\tilde{N}}$ still denoted by $(\Pi_{\tilde{N}}^{0} C_{\varepsilon}^{\tilde{N}})_{\tilde{N}}$ and $(\Pi_{\tilde{N}}^{0} N_{\varepsilon}^{\tilde{N}})_{\tilde{N}}$, and functions $N_{\varepsilon}$ and $C_{\varepsilon}$ such that
\begin{align}
&\Pi_{\tilde{N}}^{0} C_{\varepsilon}^{\tilde{N}} \rightharpoonup C_{\varepsilon} \mbox{ and } \Pi_{\tilde{N}}^{0} N_{\varepsilon}^{\tilde{N}} \rightharpoonup N_{\varepsilon} \mbox{ weakly* in } L^{\infty}(Q_{T}) \, , \label{L} \\
%\end{align}
%\begin{align}
&\Pi_{\tilde{N}}^{0} C_{\varepsilon}^{\tilde{N}} \rightharpoonup C_{\varepsilon} \mbox{ and } \Pi_{\tilde{N}}^{0} N_{\varepsilon}^{\tilde{N}} \rightharpoonup N_{\varepsilon} \mbox{ weakly in } L^{2}(0,T;H^{1}(\Omega))\, ,\label{est:c} \\
%\end{align}
%\begin{align}
&\Pi_{\tilde{N}}^{0} C_{\varepsilon}^{\tilde{N}} \longrightarrow C_{\varepsilon} \mbox{ and } \Pi_{\tilde{N}}^{0} N_{\varepsilon}^{\tilde{N}} \longrightarrow N_{\varepsilon} \mbox{ strongly in } L^{2}(Q_{T}) \mbox{ and a.e. in $Q_{T}$ },\label{imp}\\
%\end{align}
%\begin{align}
&\frac{\partial}{\partial t} \Big(\Pi_{\tilde{N}}^{1} C_{\varepsilon}^{\tilde{N}}\Big) \rightharpoonup \frac{\partial C_{\varepsilon}}{\partial t} \mbox{ and } \frac{\partial}{\partial t} \Big(\Pi_{\tilde{N}}^{1} N_{\varepsilon}^{\tilde{N}}\Big) \rightharpoonup \frac{\partial N_{\varepsilon}}{\partial t}  \mbox{ weakly in } L^{2}(0,T;(H^{1}(\Omega))')\, ,\label{derive}\\
%\end{align}
%\begin{align}
&\Pi_{\tilde{N}}^{1} C_{\varepsilon}^{\tilde{N}} \rightharpoonup C_{\varepsilon} \mbox{ and } \Pi_{\tilde{N}}^{1} N_{\varepsilon}^{\tilde{N}} \rightharpoonup N_{\varepsilon}  \mbox{ weakly* in } L^{\infty}(Q_{T})\, ,\label{L1}
\end{align}
as $\tilde{N} \longrightarrow +\infty$.\\
\noindent Moreover, $N_{\varepsilon}$ and $C_{\varepsilon}$ verify
\begin{align}
&N_{\varepsilon}(0,x)=N_{0,\varepsilon}(x) \mbox { and } C_{\varepsilon}(0,x)=C_{0,\varepsilon}(x)\,  a.e. \, ,\label{L5}\\
%\end{align}
%\begin{align}
& 0 \leq N_{\varepsilon}(t,x) \leq 1 \mbox{ and } C_{\varepsilon}(t,x) \geq 0 \mbox{ for a.e. } (t,x) \in (0,T) \times \Omega = Q_{T} \, .\label{L6}
\end{align}
 All the estimates and the compactness arguments are then easily deduced by following the same guideline of \cite{chamoun2014coupled}. The passage to the limit is very classical, we show only the convergence of the control terms
 $$
 I_h= \int_0^T \int_{\Omega_c}{ \Pi_{\tilde{N}}^{0}(f^{\tilde{N}})^{+}\ (\tau_{h} \Pi_{\tilde{N}}^{0} C_{\varepsilon}^{\tilde{N}}) \ \psi_{2}} \; dx dt,
- \int_0^T \int_{\Omega_c}{ \Pi_{\tilde{N}}^{0}(f^{\tilde{N}})^{-}\ ( \Pi_{\tilde{N}}^{0} C_{\varepsilon}^{\tilde{N}}) \ \psi_{2}} \; dx dt \,.
 $$
In fact, 
\begin{equation*}
\begin{split}
&I_h- \Big(\int_0^T\!\!\!\! \int_{\Omega_c}{ f\ C_{\varepsilon} \psi_{2}} \; dx dt \Big)
\\
&\qquad =\displaystyle\int_0^T\!\!\!\!\int_{\Omega_c}\big(\; {(\Pi_{\tilde{N}}^{0}(f^{\tilde{N}})^{+}\tau_{h}\Pi_{\tilde{N}}^{0}C_{\varepsilon}^{\tilde{N}}-f^{+}C_{\varepsilon})-(\Pi_{\tilde{N}}^{0}(f^{\tilde{N}})^{-}\Pi_{\tilde{N}}^{0}C_{\varepsilon}^{\tilde{N}}-f^{-}C_{\varepsilon})\big)
\; \psi_{2}}\; dx dt.
\\%=============================%
& \qquad \leqslant \norm  {(\Pi_{\tilde{N}}^{0}(f^{\tilde{N}})^{+} -f^{+}) \psi_{2}}_{L^2(Q_T)} \norm{\tau_{h}\Pi_{\tilde{N}}^{0}C_{\varepsilon}^{\tilde{N}}}_{L^2(Q_T)}+ \norm  {(\tau_h \Pi_{\tilde{N}}^{0}C_{\varepsilon}^{\tilde{N}} -C_{\varepsilon}) \psi_{2}}_{L^2(Q_T)} \norm{f^{+}}_{L^2(Q_T)}
\\
&\qquad 
\quad+\norm  {(\Pi_{\tilde{N}}^{0}(f^{\tilde{N}})^{-} -f^{-}) \psi_{2}}_{L^2(Q_T)} \norm{\Pi_{\tilde{N}}^{0}C_{\varepsilon}^{\tilde{N}}}_{L^2(Q_T)}+ \norm  {( \Pi_{\tilde{N}}^{0}C_{\varepsilon}^{\tilde{N}} -C_{\varepsilon}) \psi_{2}}_{L^2(Q_T)} \norm{f^{-}}_{L^2(Q_T)}.
\end{split}
\end{equation*}
By employing \eqref{V4} and \eqref{imp}, we obtain the strong convergence of $\tau_{h}\Pi_{\tilde{N}}^{0}C_{\varepsilon}^{\tilde{N}}$ to $C_{\varepsilon}$ in $L^{2}(Q_{T})$. Therefore, by using \eqref{cv fort chakel f} and since $\tau_{h}\Pi_{\tilde{N}}^{0}C_{\varepsilon}^{\tilde{N}}$ and   $\Pi_{\tilde{N}}^{0}f^{\tilde{N}}$  are in $L^2(Q_T)$, we can deduce that 
$$I_h \longrightarrow   \int_0^T\int_{\Omega_c}{ f\ C_{\varepsilon} \psi_{2}} \; dx dt .$$
%====================================================================%

Finally, using the fact that $f$ is bounded and under assumption \eqref{unicite}, we can show the uniqueness of Solution of System \eqref{CNS}-\eqref{ci} following the same techniques of \cite{chamoun2014coupled}.

\section{Optimal control problem}

This section is devoted to show the existence of optimal control and the construction of the adjoint problem using the Lagrangian
function. Furthermore, we prove the existence of a weak solution for the bilinear control problem.\\

\noindent Suppose that  $f \in\mathcal{F}=\{L^\infty([0,T]\times \Omega), \mbox{ such that }\,  \mbox{for all}\ (t,x)\in [0,T]\times \Omega,  f_{\min} \leq f(t,x)\leq f_{\max} \ $\}, % where $\Omega_c\subset \Omega$
%is the control domain
and consider $0 \leq N_0 \leq 1, C_0 \in L^{\infty}(\Omega)$ with $C_0\geq 0$ in $\Omega$. The function $f\in \mathcal{F}$ describes the bilinear control acting on the $C$-equation. We start now by proving the existence of a solution for the following optimal control problem:
\begin{equation}\label{control problem} 
   \left\{\begin{array}{l}\text { Find }(N, C, f) \in \mathcal{W}\times \mathcal{X}\times \mathcal{F} \text { minimizing the functional } \\[0.5em] \displaystyle J(N, C, f):=\frac{\gamma_N}{2} \int_0^T\left\|N(t)-N_d(t)\right\|_{L^2(\Omega)}^2 \; d t +\frac{\gamma_C}{2} \int_0^T\left\|C(t)-C_d(t)\right\|_{L^2(\Omega)}^2 \; d t+\frac{\gamma_f}{2} \int_0^T\|f(t)\|_{L^{2}\left(\Omega_c\right)}^{2}\;  d t \\[0.7em] \text { subject to }(N, C ) \text { be a weak solution of the PDE system \eqref{CNS}-\eqref{ci}.  in the sense of Definition \eqref{solution faible}} \\ \text{and}\  f\in \mathcal{F} \, ,\end{array}\right. 
\end{equation}
 
\noindent where $\mathcal{W}=\{  N \in L^2(Q_T), \, \partial_{t}N \in  L^{2}(0,T;(H^{1}( \Omega))^{'}\}$, $\mathcal{X}=\{C \in  L^{2}(0,T;H^{1}( \Omega)) \,  \partial_{t}C \in  L^{2}(0,T;(H^{1}( \Omega))')\} $,\ $\left(N_d, C_d\right) \in L^2(Q_{T})^2$ represent the target states and the non- negative numbers $\gamma_N, \gamma_C$ and $\gamma_f$ measure the cost of the states and control, respectively.

\noindent The functional $J$ defined in \eqref{control problem} characterizes  the deviation between the cell density $N$ and the chemical concentration $C$ from their respective target values $N_d$ and $C_d$ and  the cost of the control $f$ measured in the $L^2- $norm.
%========================%
\subsection{Existence of control}
In this subsection, we prove that there exists a control of \eqref{control problem} using a method inspired from \cite{braz2022bilinear}. 
\begin{theoreme}
 Assume that $0\leq  N_0\leq 1$, $\mathbf{C}_0 \in L^{\infty}\left(\Omega\right), N_d \in L^2(Q_T)$, $C_d \in L^2(Q_T)$  and $f\in \mathcal{F}$. Then, there exists a solution of the optimal control problem \eqref{control problem}. 
\end{theoreme} 

\begin{proof}
Denote the set of admissible solutions of \eqref{control problem} as 
$$ \displaystyle \mathcal{S}_{a d}=\left\{s=(N, C, f) \in \mathcal{W} \times \mathcal{X} \times \mathcal{F}: s\right.\text{ is a weak solution of}\ \eqref{CNS}-\eqref{ci}\ \mbox{in} \left.(0, T)\times\Omega\right\} \, .$$
Due to Theorem \ref{th:existence}, $\mathcal{S}_{\text {ad }}$ is non-empty. The cost functional is non-negative, and therefore it has a greatest lower bound. Thus, there exists a minimizing sequence $\left\{s_{m}\right\}_{m \in \mathbb{N}} \subset \mathcal{S}_{\text {ad }}$ such that $ \displaystyle \lim\limits _{m \rightarrow+\infty} J\left(s_{m}\right)=\inf _{s \in \mathcal{S}_{\text {ad }}} J(s)$. 

\noindent 
As $f_m \in \mathcal{F}$ and due to the definition of $\mathcal{F}$, then $f_{m}$ is bounded in $L^2(Q_T)$ and there exists a subsequence still denoted $f_m$ such that it converges weakly to a candidate solution $\tilde{f}$ in $L^2(Q_{T})$.
%is bounded 
%Due to the definition of $J$  and the assumption $\gamma_f>0$, one can obtain the boundedness of the sequence of controls $f_m$ in $L^2(Q_c)$ and hence we deduce the weak convergence of $f_m$ to a candidate solution $f$.
%
%\noindent Therefore, by taking into account that $\mathcal{F}$ is a closed convex subset of $L^{2}\left(Q_c\right)$ (hence is weakly closed in $L^{2}\left(Q_c\right)$ ), it is deduced that $ f\in\mathcal{F}$.
Furthermore, the same guidelines of \cite{chamoun2014coupled}, Section $4$ with $(N_m,C_m)$ instead of $\displaystyle (N_\varepsilon, C_\varepsilon)$ imply the strong convergence of $N_m$ and $C_m$ to a weak solution ($\tilde{N}$, $\tilde{C}$) of \eqref{CNS}-\eqref{ci} in sense of definition \ref{solution faible}. This is supported by the observation that
$$
f_{m} C_{m} 1_{\Omega_{c}}  \rightharpoonup\quad \tilde{f} \tilde{C} 1_{\Omega_{c}} \quad \text { weakly in } L^{2}(Q_{T}).
$$

%\noindent Moreover,  $\left(N_{m}(0), C_{m}(0)\right)$ converges to $\displaystyle (\tilde{N}(0), \tilde{C}(0))$ in $(H^{1}( \Omega))^{'}\times L^2(\Omega)$ , and since
%$N_{m}(0)=N_{0}, C_{m}(0)=C_{0}$, we can conclude that  $\tilde{N}(0)=N_{0}$ and $\tilde{C}(0)=C_{0}$. 
\noindent Consequently, $\tilde{s}=(\tilde{N},\tilde{C},\tilde{f})$ is a weak solution of the system \eqref{CNS}-\eqref{cl}-\eqref{ci}, %as $m$ goes to $+\infty$,
then $\tilde{s} \in \mathcal{S}_{\text {ad }}$. Hence,

\begin{equation}\label{semi cont}
  \displaystyle \lim _{m \rightarrow+\infty} J\left(s_{m}\right)=\inf _{s \in \mathcal{S}_{ad}} J(s) \leq J(\tilde{s}) . 
\end{equation}

\noindent On the other hand, since any norm of a Banach space is weakly lower semi-continuous then $J$ is lower semi-continuous on $\mathcal{S}_{a d}$. One has $\displaystyle J(\tilde{s}) \leq \lim\limits_{m \rightarrow+\infty}  \inf J\left(s_{m}\right)$. Finally, we get that $\tilde{s}$ is a global optimal control.
\end{proof}
%========================================%
%===============================================================%
%=====================================================================%
%===================================================================%
\subsection{Optimality conditions and dual problem}
In this subsection, we derive the optimality conditions based on the Lagrangian formulation defined from Equations \eqref{CNS} and \eqref{control problem} as follows:
\begin{equation*}
\begin{array}{rcl}
\medskip
L\left( N, C, p, q,\lambda, \Lambda , f\right)
&=&\displaystyle  \frac{\gamma_N}{2}\int_{Q_{T}}\left|N-N_{d}\right|^{2}\; dx  dt+\frac{\gamma_{C}}{2} \int_{Q_{T}}\left|C-C_{d}\right|^{2} \; dx  dt +\frac{\gamma_{f}}{2} \int_{0}^{T}\int_{\Omega_c}|f|^{2} \; dx  dt \\
\medskip
&& \displaystyle-\int_{Q_{T}} \partial_{t} N\  p \; dx  dt+\int_{Q_{T}} \nabla \cdot\left(a\left(N\right) \nabla N\right) p \; dx  dt -\int_{Q_{T}} \nabla \cdot\left(\chi\left(N\right) \nabla C\right) p\; dx  dt\\
\medskip
&&\displaystyle-\int_{\Omega} \lambda (N(0,x)-N_{0}) \; dx-\int_{Q_{T}} \partial_{t} C\ q \; dx dt  +\int_{Q_{T}} \Delta C\ q \; dx  dt\\
\medskip
&&\displaystyle +\int_{Q_{T}}\alpha\ N\ q \; dx  dt  -\int_{Q_{T}}\beta\ C\ q \; dx  dt + \int_{0}^{T}\int_{\Omega_c}   f\ C\ q \; dx  dt\\
\medskip
&&\displaystyle-\int_{\Omega} \Lambda (C(0,x)-C_{0}) \; dx.
\end{array}
\end{equation*}
Integrated by part and using \eqref{ci}, we get
\begin{equation*}
\begin{array}{rcl}
L\left( N, C, p, q,\lambda, \Lambda , f\right)&=& \displaystyle \frac{\gamma_N}{2}\int_{Q_{T}}\left|N-N_{d}\right|^{2} \; dx  dt+\frac{\gamma_{C}}{2} \int_{Q_{T}}\left|C-C_{d}\right|^{2} \; dx  dt +\frac{\gamma_{f}}{2} \int_{0}^{T}\int_{\Omega_c}|f|^{2} \; dx  dt \\
&& \displaystyle-\int_{Q_{T}} \partial_{t} N\ p \; dx  dt+\int_{Q_{T}} \nabla \cdot\left(a\left(N\right) \nabla p\right) N \; dx  dt-\int_{\Sigma_{T}} N \left(a\left(N\right)\nabla p\right)  \cdot \eta \; d\sigma(x)  dt \\
&&\displaystyle-\int_{Q_{T}} \nabla \cdot\left(\chi\left(N\right) \nabla p\right) C \; dx  dt+\int_{\Sigma_{T}} C \left(\chi\left(N\right)\nabla p\right)\cdot \eta  \; d\sigma(x)  dt -\int_{Q_{T}} \partial_{t} C\ q \; dxdt \\ \medskip
&& \displaystyle+\int_{Q_{T}} \Delta q\ C \; dx  dt-\int_{\Sigma_{T}}  C   \nabla q\cdot \eta \;   d\sigma(x) dt +\int_{Q_{T}}\alpha\ N\ q \; dx  dt-\int_{Q_{T}}\beta\ C\ q \; dx  dt   \\
&&\displaystyle + \int_{0}^{T}\int_{\Omega_c}   f\ C \ q \; dx  dt-\int_{\Omega} \lambda (N(0,x)-N_{0}) \; dx-\int_{\Omega} \Lambda (C(0,x)-C_{0}) \; dx.
\end{array}  
\end{equation*}
\noindent The first-order optimality system characterizing the adjoint variables is obtained by setting the partial derivatives of the Lagrangian $L\left(N, C, p, q, f\right)$ with respect to $N$ and $C$ equal to zero. Therefore, we have:
\begin{align*}
\left(\frac{\partial L}{\partial N}, \delta N\right) = \int_{Q_{T}}
\Big( \partial_{t} p &+\nabla \cdot\left(a\left(N\right) \nabla p\right)-a^{\prime}(N) \nabla N \cdot \nabla p+\chi^{\prime}(N) \nabla C \cdot \nabla p+\alpha\ q  \\
 &+ \gamma_N (N-N_{d})\Big) \delta N \; dx dt+\int_{\Omega}\big(-p(T)\delta N(T)+p(0)\delta N(0) \big)\; dx
\\
 &-\int_{\Omega} \lambda \delta N(0,x)\; dx- \int_{\Sigma_{T}}\delta N (a\left(N\right)\nabla p\cdot \eta)\;   \; d\sigma(x)  dt=0,
\end{align*}
and 
\begin{align*}
\left(\frac{\partial L}{\partial C}, \delta C\right)=  \int_{Q_{T}}\Big(\partial_{t} q &+ \Delta q-\nabla \cdot\left(\chi\left(N\right) \nabla p\right)-\beta q +fq1_{\Omega}+\gamma_{C}\left(C-C_{d}\right)\big) \delta C \; dx dt\\
 &+ \int_{\Omega}{\big(-q(T)\delta C(T)+q(0)\delta C(0)\big)}\; dx-\int_{\Omega} \Lambda \delta C(0,x)\; dx
\\
&-\int_{\Sigma_{T}}\delta C  (\nabla q\cdot \eta)\; d\sigma(x) +\int_{\Sigma_{T}} \delta C (\chi(N)\nabla p\cdot \eta ) \; dt=0. 
\end{align*}

%\noindent with boundary and final conditions

%$$
%\begin{cases}
%\displaystyle a(N)\nabla p\cdot {\eta}=(\nabla q-\chi(N)\nabla p)\cdot {\eta}=0 \text { on } %\Sigma_{T}\\
%\text { and } \\
%\displaystyle p(T)=q(T)=0 \text { on } \Omega \, .
%\end{cases}
%$$

\noindent Then we get
\begin{equation}\label{adjoint}
  \begin{cases}
\displaystyle\partial_{t} p+\nabla \cdot\left(a\left(N\right) \nabla p\right)-a^{\prime}\left(N\right) \nabla N \cdot \nabla p+\chi^{\prime}\left(N\right) \nabla C \cdot \nabla p+\alpha\ q  + \gamma_N \left(N-N_{d}\right)=0  \text { in } Q_{T}, \\[0.5em]
\displaystyle\partial_{t} q+\Delta q-\nabla \cdot\left(\chi\left(N\right) \nabla p\right)-\beta q +fq1_{\Omega}+\gamma_{C}\left(C-C_{d}\right)=0,
\end{cases}  
\end{equation}

\noindent with the following boundary and final conditions:

\begin{equation} \label{adj bord}
\begin{cases}
\displaystyle a(N)\nabla p\cdot {\eta}=(\nabla q-\chi(N)\nabla p)\cdot {\eta}=0 \text { on } \Sigma_{T}, \\
%\text {and } 
\displaystyle p(T)=q(T)=0 \text { on } \Omega \, .
\end{cases}   
\end{equation}
Moreover, we have $p(0)=\lambda$ and $q(0)=\Lambda$.

\noindent Systems \eqref{adjoint} and \eqref{adj bord} represent the adjoint problem that corresponds to the given optimal control problem \eqref{control problem}. Furthermore, the optimality condition is obtained by taking  the derivative of $L$ with respect to $f$, thus we get 
$$
\int_{0}^{T}\int_{\Omega_c} (\ \gamma_f f +C\ q\ )(\tilde{f}-f)\; dx dt\geq 0\, , \quad \forall \tilde{f}\in \mathcal{F}.
$$ 

%==================================================================================================================================%
\subsection{Existence of weak solution  of the adjoint problem 
 \eqref{adjoint}-\eqref{adj bord}}
This subsection is devoted to proving the existence of a weak solution of the adjoint system \eqref{adjoint}-\eqref{adj bord}. Initially, we observe that the equations \eqref{adjoint}-\eqref{adj bord} do not conform to the typical form of parabolic evolution equations, particularly regarding the initial condition. Instead of an initial time condition at
$t=0$, there is a final time condition at 
$t=T$. Thus,  %...... We note first that equations \eqref{adjoint}-\eqref{adj bord} do not have the usual form of parabolic evolution equations (we have $\partial _t q +\Delta q$ instead of $\partial _t q -\Delta q $ ). However, in this case, we have a final time condition at $t = T$ instead of an initial time condition at $t = 0$.
to demonstrate that problem \eqref{adjoint}-\eqref{adj bord} is well-posed, we can transform it into the standard form of a parabolic evolution equation by introducing a change of variables $s=T-t$.\\

\noindent By defining $\tilde{p}(s)=p (t),\;  \tilde{q}(s)=q (t) $, %$\tilde{N}(s)=N (t)$, %$\tilde{C}(s)=C (t)$
the problem satisfied by $\tilde{p}(s)$ and $\tilde{q}(s)$ will be as the standard form:
 \begin{equation}\label{adjoint2}
  \begin{cases}
\displaystyle\partial_{s} \tilde{p}-\nabla \cdot\left(a\left(N\right) \nabla  \tilde{p}\right)+a^{\prime}\left(N\right) \nabla N \cdot \nabla  \tilde{p}-\chi^{\prime}\left(N\right) \nabla C \cdot \nabla \tilde{p}-\alpha\  \tilde{q}  
- \gamma_N \left(N-N_{d}\right)=0 & \text { in } Q_{T},  \\[0.5em]
\displaystyle\partial_{s} \tilde{q}-\Delta  \tilde{q}+\nabla \cdot\left(\chi\left(N\right) \nabla  \tilde{p}\right)+\beta  \tilde{q}-f\tilde{q}1_{\Omega}-\gamma_{C}\left(C-C_{d}\right)=0,
\end{cases}  
\end{equation}
with following boundary and initial  conditions:

\begin{equation}
\label{adj bord2}
\begin{cases}
\displaystyle a(N)\nabla \tilde{p}\cdot{\eta}=(\nabla \tilde{q}-\chi(N)\nabla \tilde{p})\cdot {\eta}=0=0 \text { on } \Sigma_{T},\\
\displaystyle \tilde{p}(0)=\tilde{q}(0)=0 \text { on } \Omega .
\end{cases}
\end{equation}
Thus, establishing the existence of a solution for this system implies the existence of a solution for the adjoint problem\eqref{adjoint}-\eqref{adj bord}. \\
%\tony{A partir de la page 14 changer t par s et p et q par $\tilde{p}$ }
%\noindent In the sequel, we will use the same notation for $\tilde {p}$ and $p$, likewise for $\tilde {q}$ and $q$ and for t and s.\\

%========================%
\noindent Before proving the global existence of a weak solution, we need to discuss two subsections related to the concept of weak solutions for the modified adjoint system described in \eqref{adjoint2}-\eqref{adj bord2}. 
\\

\noindent The first subsection addresses a specific degenerate problem where $a(N)$ becomes zero at a finite number of points in $  (0,T)\times \Omega$, indicating that $N$ takes the values 1 and 0 at only few points in  $ t\in (0,T)\times\Omega$, to achieve the limit of $\sqrt{a(N)}\nabla \tilde{p}$ for the convergence of the non-degenerate problem. Further 
 details be provided in subsection \ref{subfinit} to establish the existence of a weak solution to problem \eqref{adjoint2}-\eqref{adj bord2} under an additional condition: the strong regularity $\sqrt{a(N}\nabla N \in L^{\infty}(Q_T)$.
\\

\noindent  Moving on to the second subsection, we introduce a new concept regarding the existence of a weak solution for a more general degenerate problem where obtaining the limit $\sqrt{a(N)}\nabla \tilde{p}$ is challenging. In this case, we use the assumptions of the first subsection, add additional conditions:  regularities of the density $N$ and the test function. \\

\noindent Now, let's start by exploring the details of the first subsection.
%===============================================================%
%==============================================================%
%==============================================================%
\subsubsection{Case of Localized Degeneracy
($a(N)$ vanishes at a finite number of points in $(0,T)\times \Omega$)}\label{subfinit}
%Finite Points and Regular Gradient}
We start now by giving the definition of a weak solution of system \eqref{adjoint2}-\eqref{adj bord2} .
%================================================%
\begin{definition}\label{def adj}
A pair  $(\tilde{p},\tilde{q})$ is a {\it{\bf weak solution}} of \eqref{adjoint2}-\eqref{adj bord2} if
\begin{align*}
&\tilde{p} \in L^{\infty}(0,T;L^{2}( \Omega)) , \quad \tilde{q} \in L^{\infty}(0,T;L^{2}( \Omega)) \cap L^{2}(0,T;H^{1}( \Omega)),
\\ & \partial_{s}\tilde{p}\in L^{\frac{2p}{p+2}}(0,T;(H^{1}( \Omega))^{\prime}), \quad \partial_{s}\tilde{q}\in L^{2}(0,T;(H^{1}( \Omega))^{\prime})
\end{align*}
\noindent and $(\tilde{p},\tilde{q}$) satisfy
\begin{equation}\label{php}
\begin{array}{ll}
\medskip
\displaystyle{ \int_{0}^T {<\partial_{s}\tilde{p}, \varphi_1>}\; dx ds + \iint_{Q_{T}}{ a(N)\nabla \tilde{p} \cdot \nabla\varphi_1} \; dx ds+\iint_{Q_{T}}{ a'(N) \nabla N\cdot \nabla \tilde{p} \varphi_1 }} \; dx ds\\
\hspace{2cm} \displaystyle {-\iint_{Q_{T}}{ \chi'(N) \nabla C\cdot \nabla \tilde{p}\ \varphi_1 } \; dx ds- \iint_{Q_{T}}{ \alpha \tilde{q} \varphi_1 } \; dx ds- \iint_{Q_{T}}{ \gamma_N(N-N_d)\varphi_1 } \; dx ds=   0} , 
\end{array}
\end{equation}
\begin{equation}\label{qhq}
\begin{array}{ll}
\medskip
\displaystyle \int_{0}^T {<\partial_{s} \tilde{q}, \varphi_2>}\; dx ds + \iint_{Q_{T}}{ \nabla \tilde{q} \cdot\nabla\varphi_1} \; dx ds-\iint_{Q_{T}}{\chi (N) \nabla \tilde{p}\cdot \nabla \varphi_2} \; dx ds+\iint_{Q_{T}}{ \beta \tilde{q} \varphi_2 } \; dx ds
\\
\hspace{2cm}\displaystyle{-\iint_{Q_{T}}{ f \tilde{q} 1_{\Omega_{c} }\varphi_2 } \; dx ds -\iint_{Q_{T}}{\gamma_C(C-C_d) \varphi_2 } \; dx ds\,=0 ,}\\
\end{array}
\end{equation}
for all $\varphi_{1} \in L^{\frac{2p}{p-2}}(0,T;H^{1}(\Omega))$ and $\varphi_{2} \in L^{2}(0,T;H^1(\Omega))$, where $(N,C)$ is a weak solution of the system \eqref{CNS}-\eqref{cl}-\eqref{ci} in the sense of Definition \ref{solution faible}.
\end{definition}
\noindent The main assumptions in this subsection are:

\begin{equation}\label{regul gradient}
\sqrt{a(N)}\nabla N  \in L^{\infty}(Q_{T}).
\end{equation}
\begin{equation}\label{nombre fini}
a(N)\ \mbox{ vanishes  at a finite number of points in }  [0,T]\times\Omega.
\end{equation}
\noindent  The principle result of this subsection is given in the following Theorem:%============================%
\begin{theoreme}\label{th exst adj}
Suppose that  $ a'(N)\leqslant \kappa\; a(N)$, where $\kappa$ is a positive constant, and there exists a function $\mu \in \mathcal{C}^{1}\left([0,1] ; \mathbb{R}^{+}\right)$, such that $\mu(u)=\frac{\chi(u)}{a(u)}$ for all $u \in(0,1)$ and $\mu(0)=\mu(1)=0$ and assume that  \eqref{regul gradient} and \eqref{nombre fini} hold. Let (N, C) be a
weak solution of equation \eqref{CNS}-\eqref{ci}. If  the hypotheses of Theorem \ref{th:existence} hold, then there exists a  solution to the modified adjoint system \eqref{adjoint2}-\eqref{adj bord2} 
%inthe sense of Definition \ref{def adj}.
\end{theoreme}
%==================================%
\noindent The proof of Theorem \ref{th exst adj} will be done with several steps: we first prove the existence of solutions of the non-degenerate problem with parameter $\varepsilon$ (by using the Faedo-Galerkin method), then we pass to the limit when $\varepsilon$ tends to $0$.    \\

\noindent We assume first that $ a'(N)\leqslant \kappa\; a(N)$, where $c$ is a  positive constant  and  that there exists a function $\mu \in \mathcal{C}^{1}\left([0,1] ; \mathbb{R}^{+}\right)$ such that $\mu(u)=\frac{\chi(u)}{a(u)}$ for all $u \in(0,1)$ and $\mu(0)=\mu(1)=0$.

\noindent Let us now introduce the following associated non-degenerate approximation
of system \eqref{adjoint2}-\eqref{adj bord2} for a fixed $\varepsilon$ where $a_{\varepsilon}=a(N)+\varepsilon$:
\begin{equation}\label{adj non dege}
  \begin{cases}
\displaystyle\partial_{s} \tilde{p}_{\varepsilon}-\nabla \cdot\left(a_{\varepsilon}\left(N\right) \nabla  \tilde{p}_{\varepsilon}\right)+{V}\cdot \nabla  \tilde{p}_{\varepsilon}=\alpha \tilde{q}_{\varepsilon}
+\gamma_N \left(N-N_{d}\right) \quad \text { in } Q_{T}, \\[0.5em]
\displaystyle\partial_{s} \tilde{q}_{\varepsilon}-\Delta  \tilde{q}_{\varepsilon}+\nabla \cdot\left(\chi\left(N\right) \nabla  \tilde{p}_{\varepsilon}\right)+\beta  \tilde{q}_{\varepsilon}-f \tilde{q}_{\varepsilon}1_{\Omega} =\gamma_{C}\left(C-C_{d}\right),
\end{cases}  
\end{equation}
with the following boundary and initial conditions:

\begin{equation}
\label{adj non  bord }
\begin{cases}
\displaystyle a_{\varepsilon}(N)\nabla \tilde{p}_{\varepsilon}\cdot {\eta}=(\nabla \tilde{q}_{\varepsilon}-\chi(N)\nabla \tilde{p}_{\varepsilon})\cdot {\eta}=0=0 \text { on } \Sigma_{T}, \\
\displaystyle \tilde{p}_{\varepsilon}(0)=\tilde{q}_{\varepsilon}(0)=0 \text { on } \Omega \, ,
\end{cases}\\
\end{equation}
where $\  {V}= a'(N)\nabla N-\chi'(N)\nabla C$. \\
\noindent We recall that the proof of the existence is based on the assumption $\nabla N \in L^{\infty}(Q_T)$, hence  ${V}\in L^{\infty}(Q_T)$. \\
%============================================================%

\noindent The next proposition states the existence of a weak solution of the problem \eqref{adj non dege}-\eqref{adj non bord }.
%\noindent Theorem \ref{th exst adj} is then  proved using a multi-step approach. We first prove existence of solutions of the non-degenerate problem \eqref{adj non dege}-\eqref{adj non bord }, by using the Faedo-Galerkin method, in the sense of the following Proposition:
%===============================================================================%
\begin{proposition}\label{prop adj exis}
Under the assumptions \eqref{regul gradient} and \eqref{nombre fini},  the non-degenerate problem \eqref{adj non dege}-\eqref{adj non bord } admits at least one weak solution $(\tilde{p}_{\varepsilon}, \tilde{q}_{\varepsilon})$, %in the sense of Definition \ref{solution faible},
such that:\\
\noindent $\forall \varphi_{1}\in L^{\frac{2p}{p-2}}(0,T;H^{1}({\Omega})) \text{ and } \ \varphi_{2} \in L^{2}(0,T;H^{1}({\Omega}))$,  
      $(\tilde{p}_{\varepsilon},\tilde{q}_{\varepsilon}$) satisfy
\begin{equation}\label{phpepsi}
\begin{split}
   &\hspace{-2.5em}\displaystyle{ \int_{0}^T {<\partial_{s}\tilde{p}_{\varepsilon}, \varphi_1>}\; dx ds + \iint_{Q_T}{ a_{\varepsilon}(N)\nabla \tilde{p}_{\varepsilon} \cdot\nabla\varphi_1} \; dx ds+\iint_{Q_{T}}{ a'(N) \nabla N\cdot\nabla \tilde{p}_\varepsilon \varphi_1 } \; dx ds}\\[0.1em]
   &\hspace{1cm} \displaystyle{-\iint_{Q_{T}}{ \chi'(N) \nabla C\cdot \nabla \tilde{p}_{\varepsilon}\ \varphi_1 } \; dx ds- \iint_{Q_{T}}{ \alpha \tilde{q}_{\varepsilon} \varphi_1 } \; dx ds- \iint_{Q_{T}}{ \gamma_N(N-N_d)\varphi_1 } \; dx ds=   0} , 
\end{split}
\end{equation}
\begin{equation}\label{qhqepsi}
\begin{split}
& \hspace{-8em}\quad \int_{0}^T {<\partial_{s}\tilde{q}_{\varepsilon}, \varphi_2>}\; dx ds + \iint_{Q_{T}}{ \nabla \tilde{q}_{\varepsilon} \cdot\nabla\varphi_2} \; dx ds-\iint_{Q_{T}}{\chi (N) \nabla \tilde{p}_{\varepsilon}\cdot \nabla \varphi_2} \; dx ds+\iint_{Q_{T}}{ \beta \tilde{q}_{\varepsilon} \varphi_2 } \; dx ds
\\[0.1em]
&\hspace{-2em}\displaystyle{-\iint_{Q_{T}}{ f \tilde{q}_{\varepsilon} 1_{\Omega_{c} }\varphi_2 } \; dx ds -\iint_{Q_{T}}{\gamma_C(C-C_d) \varphi_2 } \; dx ds\,=0 .}\\
 \end{split}
\end{equation}
\end{proposition}
%==================================================%
\noindent
\begin{proof}
To prove the existence of at least one solution of System \eqref{phpepsi}-\eqref{qhqepsi}, we will use the Faedo-Galerkin method, deriving a priori estimates and passing to the limit by using compactness arguments.
%\noindent Convergence is well achieved by means of a priori estimates. Finally, we gradually tend the regularization parameter $\varepsilon$ to zero, in order to obtain a weak solution of the adjoint system \eqref{adjoint2}-\eqref{adj bord2} as defined in Definition \ref{def adj}. This convergence is established by utilizing a priori estimates and compactness arguments, and thus the weak solution is obtained as the limit of a sequence of such approximated solutions.
%===========================%

%\subsubsection{Existence of a  solution to the non-degenerate adjoint system \eqref{phpepsi}-\eqref{qhqepsi}}
%The objective of this section is to prove the existence of a weak solution for  non-degenerate approximation systems. 
%===================================%

\noindent To use the Faedo-Galerkin approximation method, we consider an appropriate spectral problem introduced in \cite{sowndarrajan2018optimal} and \cite{bendahmane2005analysis}.  The eigenfunctions $e_l(x)$, identified by the introduced spectral problem, form an orthogonal basis in $H^1(\Omega)$ and an orthonormal basis in $L^2(\Omega)$.%===================================================%

\noindent Our aim is to identify the finite-dimensional approximation solutions for system \eqref{adj non dege}-\eqref{adj non  bord } as sequences $\left\{\tilde{p}_{n, \varepsilon}\right\}_{n>1}$ and $\left\{\tilde{q}_{n, \varepsilon}\right\}_{n>1}$ defined for $s \geq 0$ and $x \in \bar{\Omega}$ by%==============%
$$ 
\tilde{p}_{n,\varepsilon}(s,x)=\sum_{l=1}^n c_{ n, l}(s) e_l(x),\quad 
\tilde{q}_{n,\varepsilon}(s,x)=\sum_{l=1}^n d_{ n, l}(s) e_l(x),
$$ %========================================================% 
with the initial conditions
$$
\tilde{p}_{n,\varepsilon}(0,x)=\tilde{q}_{n,\varepsilon}(0,x)=0 .
$$%================================================================%

\noindent Further, it should be noted that the above form of solutions satisfies the required boundary conditions. \\
\noindent Next, we have to determine the  coefficients $\left\{c_{ n, l}(s)\right\}_{l=1}^n$ and $\left\{d_{ n, l}(s)\right\}_{l=1}^n$ such that for $l=1, \ldots, n,$

\begin{equation}\label{phpn}
\begin{split}
   &\hspace{-2.5em}\displaystyle{\int_{\Omega} {\partial_{s}\tilde{p}_{n,\varepsilon}\ e_l}\; dx +  \int_{\Omega}{ a_{\varepsilon}(N)\nabla \tilde{p}_{n,\varepsilon}\cdot \nabla e_l} \; dx + \int_{\Omega}{ {V}\cdot \nabla \tilde{p}_{n,\varepsilon} e_l } \; dx -  \int_{\Omega}{ \alpha\ \tilde{q}_{n,\varepsilon} e_l} \; dx }\\[0.1em]
   &\displaystyle{-\int_{\Omega}{ \gamma_N(N-N_d)e_l } \; dx =   0} , 
\end{split}
\end{equation}
\begin{equation}\label{qhqn}
\begin{split}
& \hspace{-6em}\quad \int_{\Omega} {\partial_{s}\tilde{q}_{n,\varepsilon}\; e_l}\; dx  +  \int_{\Omega}{ \nabla \tilde{q}_{n,\varepsilon} \cdot\nabla e_l} \; dx - \int_{\Omega}{\chi (N) \nabla \tilde{p}_{n,\varepsilon}\cdot \nabla e_l} \; dx +\int_{\Omega}{ \beta\ \tilde{q}_{n,\varepsilon} e_l } \; dx 
\\[0.1em]
&\hspace{-2em}\displaystyle{- \int_{\Omega}{ f \tilde{q}_{n,\varepsilon} 1_{\Omega_{c} }e_l } \; dx  -\int_{\Omega}{\gamma_C(C-C_d) e_l } \; dx \,=0 .}\\
 \end{split}
\end{equation}
%and  with reference to the initial conditions:
%$$
%\begin{cases}
  %\displaystyle  \tilde{p}_{n,\varepsilon}(0, x)=0, %\\
% \displaystyle \tilde{q}_{n,\varepsilon}(0, x)=0.
%\end{cases}
%$$%========================================================%

\noindent More explicitly, by using the orthonormality of the basis, we can write \eqref{phpn} and \eqref{qhqn} as a system of  ordinary differential equations (l=1,...,n)%==========================%
\begin{align}
%\begin{split}
c'_{ n, l}(s)=&-\displaystyle \int_{\Omega}{ a_{\varepsilon}(N)\nabla \tilde{p}_{n,\varepsilon} \cdot\nabla e_l} \; dx -\int_{\Omega}{ {V}\cdot \nabla \tilde{p}_{n,\varepsilon} e_l } \; dx +\int_{\Omega}{ \alpha\ \tilde{q}_{n,\varepsilon} e_l} \; dx + \int_{\Omega}{ \gamma_N(N-N_d)e_l } \; dx \notag\\
     := & F_{1,l}\left(s, c_{n, 1}(s), \ldots, c_{n, n}(s)\right), \label{phpn3}\\ 
%\end{split}
%\end{equation}
%\begin{equation}
%\begin{split}
\hspace{2em}d'_{ n, l}(s)=&- \displaystyle\int_{\Omega}{\nabla \tilde{q}_{n,\varepsilon}\cdot \nabla e_l} \; dx + \int_{\Omega}{ \chi (N) \nabla \tilde{p}_{n,\varepsilon} \cdot\nabla e_l} \; dx \,- \int_{\Omega}{ \beta\ \tilde{q}_{n,\varepsilon}\ e_l } \; dx \,+ \int_{\Omega}{ f \tilde{q}_{n,\varepsilon} e_l 1_{\Omega} } \; dx\, \notag\\
 &+\int_{\Omega}{\gamma_C(C-C_d) e_l } \; dx \, \notag
\\:= & F_{2,l}\left(s, c_{n, 1}(s), \ldots, c_{n, n}(s)\right).\label{qhqn3}  
%\end{split}
\end{align}
%================================%
 Observe that $F_{i,l}\left(s, c_{n, 1}(s), \ldots, c_{n, n}(s)\right)$, $i=1,2$, are Carath\'eodory functions. Therefore, using the standard ODE theory, there exists absolutely continuous functions $\left\{c_{n, l}\right\}_{l=1}^{n}$, $\left\{d_{n, l}\right\}_{l=1}^{n}$ satisfying \eqref{phpn3} and \eqref{qhqn3} respectively for a.e. $s\in\left[0, s^{\prime}\right)$ for some $s^{\prime}>0$. This proves that the sequences $\left(\tilde{p}_n^{\varepsilon}, \tilde{q}_n^{\varepsilon}\right)$ are well defined and the approximate solution of Equations \eqref{phpn3} and \eqref{qhqn3} exists on $\left(0, s^{\prime}\right)$.%================%

\noindent Let $\displaystyle\phi_{i, n}(s,x)=\sum_{l=1}^n b_{i, n, l}(s) e_l(x), i=1, 2$, where the coefficients $\left\{b_{i, n, l}\right\}, i=1, 2$, are absolutely continuous functions. Then the approximate solution satisfies the weak formulation:%=============%

\begin{align}
%\begin{split}
&\displaystyle{\int_{\Omega} {\partial_{s}\tilde{p}_{n,\varepsilon} \phi_{1,n}}\; dx +  \int_{\Omega}{ a_{\varepsilon}(N)\nabla \tilde{p}_{n,\varepsilon}\cdot \nabla \phi_{1,n}} \; dx + \int_{\Omega}{ {V}\cdot \nabla \tilde{p}_{n,\varepsilon} \phi_{1,n} } \; dx -  \displaystyle\int_{\Omega}{ \alpha\ \tilde{q}_{n,\varepsilon} \phi_{1,n}} \; dx } \notag \\
& \hspace{6cm}\displaystyle{-\int_{\Omega}{ \gamma_N(N-N_d)\phi_{1,n} } \; dx =   0} ,\label{phpfin} \\
%\end{split}
%&\end{align}
%&\begin{equation}
%&\begin{split}
& \int_{\Omega} {\partial_{s}\tilde{q}_{n,\varepsilon}\; \phi_{2,n}}\; dx  +  \int_{\Omega}{ \nabla \tilde{q}_{n,\varepsilon} \cdot\nabla \phi_{2,n}} \; dx - \int_{\Omega}{\chi (N) \nabla\tilde{p}_{n,\varepsilon} \cdot\nabla \phi_{2,n}} \; dx + \int_{\Omega}{ \beta \tilde{q}_{n,\varepsilon} \phi_{2,n} } \; dx \notag \\
&\hspace{6cm} - \displaystyle \int_{\Omega}{ f \tilde{q}_{n,\varepsilon} 1_{\Omega_{c} }\phi_{2,n} } \; dx - \int_{\Omega}{\gamma_C(C-C_d) \phi_{2,n} } \; dx \,=0\; . \label{qhqfin}
\end{align}
%===============================================================%

\noindent Taking respectively  $\phi_{2,n}=\tilde{q}_{n,\varepsilon}$ and  $\phi_{1,n}=\tilde{p}_{n,\varepsilon}$  in \eqref{phpfin} and \eqref{qhqfin}, we have
\begin{equation}\label{est q epsi}
\begin{split}
    \frac{1}{2} \frac{d}{d s}\left\|\tilde{q}_{n,\varepsilon}\right\|_{L^2(\Omega)}^2+\frac{1}{2}\left\|\nabla \tilde{q}_{n,\varepsilon}\right\|_{L^2(\Omega)}^2 \leqslant &\frac{1}{2}\|\chi(N)\|_{L^{\infty}(\Omega)}^2\left\|\nabla \tilde{p}_{n,\varepsilon}\right\|_{L^2(\Omega)}^2+(\beta+\frac{1}{2}+\|f\|_{L^{\infty}(\Omega)})\left\|\tilde{q}_{n,\varepsilon}\right\|_{L^2(\Omega)}^2 \\
    &+\frac{\gamma_C^2}{2 }\ \norm{C-C_d}_{L^2(\Omega)}^2 \, 
\end{split}
\end{equation}
%==========================================================%
 and 
\begin{equation}\label{p eps est}
    \begin{split}
       &\frac{1}{2} \frac{d}{ds }\left\|\tilde{p}_{n,\varepsilon}\right\|_{L^2(\Omega)}^2+\left\|\sqrt{a(N)} \nabla \tilde{p}_{n,\varepsilon}\right\|_{L^2(\Omega)}^2+\varepsilon\left\|\nabla \tilde{p}_{n,\varepsilon}\right\|_{L^2(\Omega)}^2  \\
       &\quad\leqslant-\int_{\Omega} {V} \cdot \nabla \tilde{p}_{n,\varepsilon} \tilde{p}_{n,\varepsilon}\; dx+\frac{\alpha^2}{2}\left\|\tilde{q}_{n,\varepsilon}\right\|_{L^2(\Omega)}^2+\frac{\gamma_N^2}{2}\|N-N_d\|_{L^2(\Omega)}^2  +\left\|\tilde{p}_{n,\varepsilon}\right\|_{L^2(\Omega)}^2  \\
       &\quad\leqslant \frac{1}{2 \varepsilon}\left\|{V}\right\|_{L^{\infty}(Q_T)}\left\|\tilde{p}_{n,\varepsilon}\right\|_{L^2(\Omega)}^2+\frac{\varepsilon}{2}\left\|\nabla \tilde{p}_{n,\varepsilon}\right\|_{L^2(\Omega)}^2
       \\
       &\qquad+\frac{\alpha^2}{2}\left\|\tilde{q}_{n,\varepsilon}\right\|_{L^2(\Omega)}^2+\frac{\gamma_ N^2}{2}\|N-N\|_{L^2(\Omega)}^2+\left\|\tilde{p}_{n,\varepsilon}\right\|_{L^2(\Omega)}^2.  
    \end{split}
\end{equation}
%==============================% 

\noindent Let's multiply now  \eqref{est q epsi}  by $\delta$ and then sum \eqref{est q epsi} and \eqref{p eps est}. This yields
$$
\begin{aligned}
& \frac{1}{2} \frac{d}{d s}\left(\left\|\tilde{p}_{n,\varepsilon}\right\|_{L^{2}(\Omega)}^{2}+\delta\left\|\tilde{q}_{n,\varepsilon}\right\|_{L^{2}(\Omega)}^{2}\right)+\frac{1}{2} \delta\left\|\nabla \tilde{q}_{n,\varepsilon}\right\|_{L^{2}(\Omega)}^{2} +\left\|\sqrt{a(N)} \nabla \tilde{p}_{n,\varepsilon}\right\|_{L^{2}(\Omega)}^{2}\\
&+\frac{1}{2}\left(\varepsilon-\delta\|\chi(N)\|_{L^{\infty}(\Omega)}^{2}\right)\left\|\nabla \tilde{p}_{n,\varepsilon}\right\|_{L^{2}(\Omega)}^{2}   \leq \frac{1}{2\varepsilon}\left\|{V}\right\|_{L^{\infty}(Q_T)}\left\|\tilde{p}_{n,\varepsilon}\right\|_{L^{2}(\Omega)}^{2}+\frac{\alpha^{2}}{2}\left\|\tilde{q}_{n,\varepsilon}\right\|_{L^{2}(\Omega)}^{2}  +\frac{\gamma_{N}^2}{2}\|N-N_d\|_{L^{2}(\Omega)}^{2}\\
&\qquad\qquad\hspace{2cm}\qquad\qquad\qquad+\left\|\tilde{p}_{n,\varepsilon}\right\|_{L^{2}(\Omega)}^{2}+\delta(\beta+\frac{1}{2}+\|f\|_{L^{\infty}(\Omega)})\left\|\tilde{q}_{n,\varepsilon}\right\|_{L^{2}(\Omega)}^{2} +\delta\frac{\gamma_{C}^{2}}{2 }\|C-C_d\|_{L^{2}(\Omega)}^{2}.
\end{aligned}
$$

\noindent For a fixed $\varepsilon$, we choose a small $\delta > 0$ such that $\left(\varepsilon-\delta\|\chi(N)\|_{L^{\infty}(\Omega)}^{2}\right)$ is positive. Then integrating over $(0,s)$, using the assumption \eqref{regul gradient} and  Gronwall inequalities, we can deduce that%=================================%

$$
\displaystyle \big(\tilde{p}_{n,\varepsilon}, \tilde{q}_{n,\varepsilon}\big) \in \big(L^{\infty}\left(0, s ; L^{2}(\Omega)\right) \displaystyle\cap L^{2}\left(0, s ; H^{1}\left(\Omega\right)\right)\big)^{2} \, .$$
Moreover, one can show that%=================================%
$$
\displaystyle\norm{\left(\partial_s \tilde{p}_{n,\varepsilon}, \partial_s \tilde{q}_{n,\varepsilon}\right)}_{{L^{\frac{2p}{p+2}}\left(0, s, (H^{1}(\Omega))^{\prime}\right)}\times{L^2\left(0, s, (H^{1}(\Omega))^{\prime}\right)}} \leq A,
$$
where the constant $A$ is independent of $n$. 
Hence, the local solution constructed above can be extended to the whole time interval $[0, T]$  as in \cite{bendahmane2009convergence}, so we omit the details.

\noindent Therefore as $n \rightarrow +\infty$, we have

\begin{equation}
\begin{array}{rccl}
\left(\tilde{p}_{n,\varepsilon}, \tilde{q}_{n,\varepsilon}\right) &\rightharpoonup&\left(\tilde{p}_{\varepsilon}, \tilde{q}_{\varepsilon}\right)  & \text{ weakly* in } L^{\infty}\left(0, T ; L^2(\Omega)\right), \\
\left(\tilde{p}_{n,\varepsilon}, \tilde{q}_{n,\varepsilon}\right)  &\rightharpoonup&\left(\tilde{p}_{\varepsilon}, q_{\varepsilon}\right) & \text{ weakly in } L^2\left(0, T ; H^1(\Omega)\right), \\
\left(\nabla \tilde{p}_{n,\varepsilon}, \nabla \tilde{q}_{n,\varepsilon}\right)  &\rightharpoonup&\left(\nabla \tilde{p}_{\varepsilon}, \nabla \tilde{q}_{\varepsilon}\right) & \text{ weakly in } L^2\left(Q_T\right), \\
\left(\partial_s \tilde{p}_{n,\varepsilon}, \partial_s \tilde{q}_{n,\varepsilon}\right)  &\rightharpoonup&\left(\partial_s \tilde{p}_{\varepsilon}, \partial_s \tilde{q}_{\varepsilon}\right) & \text{ weakly in } {{L^{\frac{2p}{p+2}}\left(0, s, (H^{1}(\Omega))^{\prime}\right)}\times{L^2\left(0, s, (H^{1}(\Omega))^{\prime}\right)}}  .
\end{array}
\end{equation}

\noindent One can easily now deduce the weak convergence of $\tilde{p}_{n,\varepsilon} $  and $\tilde{q}_{n,\varepsilon}$ to $\tilde{p}_{\varepsilon}$  and  $\tilde{q}_{\varepsilon}$ respectively, satisfying the definition of a weak solution of the non-degenerate system \eqref{adj non dege}-\eqref{adj non  bord }. 
This achieves the proof of the proposition \ref{prop adj exis}.
\end{proof}

\noindent To show the existence of a weak solution to Problem \eqref{adjoint2}-\eqref{adj bord2} in the sense of definition \ref{def adj}, we still have to pass to the limit in System \eqref{phpepsi}-\eqref{qhqepsi} when $\varepsilon$ tends to $0$. 
%========================================================================================================================================================================%
 %\subsubsection{Weak solution of the degenerate adjoint problem \eqref{php}-\eqref{qhq}}\label{limite}
%\noindent The aim of this subsection  is to pass to the limit in System \eqref{phpepsi}-\eqref{qhqepsi} (as $\varepsilon$ tends to $0$) in order to show the existence of a  solution to Problem \eqref{php}-\eqref{qhq}. 
%take the regularization parameter $\varepsilon$ to zero in sequences of weak solutions of problem \eqref{adj non dege}-\eqref{adj non bord } to obtain a weak solution of the original system \eqref{adjoint2}-\eqref{adj bord2} in the sense of Definition \ref{def adj} undet the assumption \eqref{regul gradient} and \eqref{nombre fini}. It is important to note that, for each fixed $\varepsilon>0$, we have demonstrated the existence of a solution $\left(p_{\varepsilon}, q_{\varepsilon}\right)$. Then to conclude on the existence of a weak solution of \eqref{adjoint2},

\noindent First, we shall need to prove the following uniform a priori estimates:\\
we replace  $\chi(N)$ by $\mu(N)a(N)$, use $a'(N) \le \kappa \ a(N)$, choose  $\tilde{q}_\varepsilon$ as a test function in \eqref{qhqepsi} and use Young's inequalities to obtain:
\begin{equation*}
    \begin{array}{ll}
    \medskip       &\hspace{-1cm}\displaystyle\frac{1}{2}\int_{0}^T  \frac{d}{d s}\left\|\tilde{q}_{\varepsilon}(s)\right\|_{L^{2}(\Omega)}^{2}\; ds+
       \iint_{Q_{T}}
      \abs{\nabla \tilde{q}_{\varepsilon}(s)}^{2}\; dxds \\
      \medskip
      &\hspace{-1cm}\displaystyle\leqslant  \iint_{Q_{T}} \left(a(N)^{\frac{1}{2}} \nabla \tilde{p}_{\varepsilon}\right)\left(a(N)^{\frac{1}{2}}\mu(N)\right) \cdot \nabla \tilde{q}_{\varepsilon}\; dxds+\left(\beta+\frac{1}{2}+\|f\|_{L^{\infty}(Q_{T})}\right)\int_{0
}^{T}\left\|\tilde{q}_{\varepsilon}\right\|_{L^{2}(\Omega)}^{2}\;ds \\
\medskip
&\hspace{-1cm}\quad \displaystyle +\int_{0}^{T}\frac{\gamma^{2}_C}{2 }\|C(s)-C_d(s)\|_{L^{2}(\Omega)}^{2}\; ds
\\
\medskip
%&\hspace{-1cm} \displaystyle\leqslant \frac{1}{2}  \norm{\sqrt{a(N)}}_{L^{\infty}(Q_{T})}^2\norm{\mu(N)}_{L^{\infty}(Q_{T})}^2
% \int_{0}^{T}\left\|\sqrt{a(N)} \nabla \tilde{p}_{\varepsilon}\right\|_{L^{2}(\Omega)}^{2}\; ds
% +\frac{1}{2}\int_{0}^{T}
%        \left\|\nabla \tilde{q}_{\varepsilon}\right\|_{L^{2}(\Omega)}^{2}\;ds
%\\
%\medskip
%&\hspace{-1cm}\quad
%\displaystyle+\left(\beta+\frac{1}{2}+\|f\|_{L^{\infty}(Q_{T})}\right)\int_{0
%}^{T}\left\|\tilde{q}_{\varepsilon}\right\|_{L^{2}(\Omega)}^{2}\;ds  \displaystyle +\int_{0}^{T}\frac{\gamma^{2}_C}{2 }\|C(s)-C_d(s)\|_{L^{2}(\Omega)}^{2}\; ds\\
 
&\hspace{-1cm} \displaystyle\leqslant \frac{1}{2}  \norm{a(N)}_{L^{\infty}(Q_{T})}\norm{\mu(N)}_{L^{\infty}(Q_{T})}^2
\int_{0}^{T}\left\|\sqrt{a(N)} \nabla \tilde{p}_{\varepsilon}\right\|_{L^{2}(\Omega)}^{2}\; ds
 +\frac{1}{2}\int_{0}^{T}
        \left\|\nabla \tilde{q}_{\varepsilon}\right\|_{L^{2}(\Omega)}^{2}\;ds
\\[1em]
&\hspace{-1cm}\quad
\displaystyle+\left(\beta+\frac{1}{2}+\|f\|_{L^{\infty}(Q_{T})}\right)\int_{0
}^{T}\left\|\tilde{q}_{\varepsilon}\right\|_{L^{2}(\Omega)}^{2}\;ds  \displaystyle +\int_{0}^{T}\frac{\gamma^{2}_C}{2 }\|C(s)-C_d(s)\|_{L^{2}(\Omega)}^{2}\; ds.

\end{array}
\end{equation*}

\noindent Then, we have

\begin{equation}\label{est q eps}
\begin{array}{ll}
\medskip
  &\hspace{-1cm}\displaystyle\frac{1}{2}\int_{0}^T  \frac{d}{d s}\left\|\tilde{q}_{\varepsilon}(s)\right\|_{L^{2}(\Omega)}^{2}\; ds+
       \frac{1}{2}\iint_{Q_{T}}
      \abs{\nabla \tilde{q}_{\varepsilon}(s)}^{2}\; dxds \\
      \medskip
      &\hspace{-1cm}\displaystyle\leqslant \frac{1}{2}  \norm{a(N)}_{L^{\infty}(Q_{T})}\norm{\mu(N)}_{L^{\infty}(Q_{T})}^2
 \int_{0}^{T}\left\|\sqrt{a(N)} \nabla \tilde{p}_{\varepsilon}\right\|_{L^{2}(\Omega)}^{2}\; ds
\\[1em]
&\hspace{-1cm}\quad
\displaystyle+\left(\beta+\frac{1}{2}+\|f\|_{L^{\infty}(Q_{T})}\right)\int_{0
}^{T}\left\|\tilde{q}_{\varepsilon}\right\|_{L^{2}(\Omega)}^{2}\;ds  \displaystyle +\int_{0}^{T}\frac{\gamma^{2}_C}{2 }\|C(s)-C_d(s)\|_{L^{2}(\Omega)}^{2}\; ds.
\end{array}  
\end{equation}

\noindent Now, choosing $\tilde{p}_\varepsilon$ as a test function in \eqref{phpepsi}, we get:

\begin{equation}\label{pteps}
\begin{split}
    \int_{0}^{T}\left(\frac{1}{2} \frac{d}{d s}\left\|\tilde{p}_{\varepsilon}\right\|_{L^{2}(\Omega)}^{2}+\left\|\sqrt{a(N)} \nabla \tilde{p}_{\varepsilon}\right\|_{L^{2}(\Omega)}^{2}+\varepsilon\left\|\nabla \tilde{p}_{\varepsilon}\right\|_{L^{2}(\Omega)}^{2}\right)\;ds 
    &\leqslant-  \iint_{Q_{T}} {V} \cdot \nabla \tilde{p}_{\varepsilon} \tilde{p}_{\varepsilon}\; dxds
    \\
&\hspace{-3cm}+\int_{0}^{ T}\left(\frac{\alpha^{2}}{2}\left\|\tilde{q}_{\varepsilon}\right\|_{L^{2}(\Omega)}^{2} +\left\|\tilde{p}_{\varepsilon}\right\|_{L^{2}(\Omega)}^{2} +\frac{\gamma_{N}^{2}}{2}\|N-N_d\|_{L^{2}(\Omega)}^{2}\right)\;ds.
\end{split}
\end{equation}

\noindent As ${V}=a'(N)\nabla N -\chi^{\prime}(N) \nabla C$ and $a'(N)\leqslant \kappa \; a(N)$, one has

%we use the assumption \mazen{$\sqrt{a(N)}\nabla N\in L^{\infty}(Q_T)$} in \eqref{regul gradient} to conclude that

\begin{equation}\label{vecteur}
\begin{split}
  &\left| \iint_{Q_T} {V} \cdot \nabla \tilde{p}_{\varepsilon} \tilde{p}_{\varepsilon}\;dxds\right| \\&\leqslant\left| \iint_{Q_T} \kappa\; a(N)\nabla N \cdot  \nabla \tilde{p}_{\varepsilon} \tilde{p}_{\varepsilon}\;dxds\right|+\left|\iint_{Q_T} \chi^{\prime}(N) \ \nabla C \cdot \nabla \tilde{p}_{\varepsilon} \tilde{p}_{\varepsilon}\;dxds\right| \\%======%
  &\leqslant  \left|\iint_{Q_T}\kappa(\sqrt{a(N)} \nabla \tilde{p}_{\varepsilon}) \cdot \left( (\sqrt{a(N)} \nabla N) \tilde{p}_{\varepsilon}\right)\;dxds\right|+\left|\iint_{Q_T} {(\mu^{\prime}(N) \sqrt{a(N)} )\nabla C \cdot (\sqrt{a(N)}\nabla \tilde{p}_{\varepsilon} )\tilde{p}_{\varepsilon}}\; dxds\right|\\%====%
  &+\left| \iint_{Q_T} {\kappa (\sqrt{a(N)}\mu(N))
 \nabla C \cdot(\sqrt {a(N)} \nabla \tilde{p}_{\varepsilon}) \tilde{p}_{\varepsilon}}\; dxds\right|.
 \end{split}
 \end{equation}
Using Young's inequality, we get 
 \begin{equation}
 \begin{split}
   &\left| \iint_{Q_T} {V} \cdot \nabla \tilde{p}_{\varepsilon} \tilde{p}_{\varepsilon}\;dxds\right|\leqslant  \frac{1}{8}\int_{0}^{T}\left\|\sqrt{a(N)} \nabla \tilde{p}_{\varepsilon}\right\|_{L^{2}(\Omega)}^{2}\;ds + 2 \kappa^{2} \norm{ \sqrt{a(N)}\nabla N}_{L^{\infty}(Q_{T})}^{2}\int_{0}^{T}\norm{\tilde{p}_\varepsilon}_{L^{2}(\Omega)}^2\;ds  \\%===%
 & +\frac{1}{8}\int_{0}^{T}\left\|\sqrt{a(N)} \nabla \tilde{p}_{\varepsilon}\right\|_{L^{2}(\Omega)}^{2}\;ds+4\left\|\mu^{\prime}(N)\right\|_{L^{\infty}(Q_{T})}^2\|a(N)\|_{L^{\infty}(Q_{T})}\int_{0}^{T}\|\nabla C\|_{L^{\infty}(\Omega)}\left\|\tilde{p}_{\varepsilon}\right\|_{L^{2}(\Omega)}^2\;ds\\
 & +\left\|\mu(N)\right\|_{L^{\infty}(Q_T)}^2\|a(N)\|_{L^{\infty}(Q_T)}\int_{0}^{T}4\kappa^2\|\nabla C\|_{L^{\infty}(\Omega)}^2\left\|\tilde{p}_{\varepsilon}\right\|_{L^{2}(\Omega)}^2\;ds. 
 \end{split}
 \end{equation}

\noindent Now, we recall from \cite{ladyzhenskaia1968linear} the regularity properties of the parabolic equation governing the concentration:

\begin{equation}\label{Cregularity}
\nabla C \in L^{p}(0,T; L^{\infty}(\Omega)), \quad p \geq 2.
\end{equation}

%Now, we recall the regularity of the parabolic equation on the concentration: 
% \begin{equation}\label{Cregularity}
% \nabla C \in L^{p}(0,T;L^{\infty}(\Omega)), \quad p\geq 2.
% \end{equation}
% We refer to \cite{ladyzhenskaia1968linear} to classical parabolic regularity results. 
\noindent Additionally, using the assumption  \eqref{regul gradient} on the cell density, we get    
\\%============%
\begin{equation}
    \begin{split}
&\left| \iint_{Q_T} {V} \cdot \nabla \tilde{p}_{\varepsilon} \tilde{p}_{\varepsilon}\;dxds\right|\leqslant  \frac{1}{4}\int_{0}^{T}\left\|\sqrt{a(N)} \nabla \tilde{p}_{\varepsilon}\right\|_{L^{2}(\Omega)}^{2}\;ds+\int_{0}^{T} K\left\|\tilde{p}_{\varepsilon}\right\|_{L^{2}(\Omega)}^{2}\;ds, 
\end{split}    
\end{equation}
where $K=2 \kappa^{2}  \norm{  \sqrt{a(N)}\nabla N }_{L^{\infty}(Q_{T})}^{2}+(4\left\|\mu^{\prime}(N)\right\|_{L^{\infty}(Q_{T})}^2+4\kappa^2\left\|\mu(N)\right\|_{L^{\infty}(Q_{T})}^2) \norm{a(N)}_{L^{\infty}(Q_{T})}\norm{\nabla C}_{L^{\infty}(\Omega)}^{2}$ is 
 independent of $\varepsilon$. 
 \\
 
 \noindent Consequently, inequality \eqref{pteps} leads to
\begin{multline}
    \label{pesp estm}
       \displaystyle\int_{0}^{T}{ \left(\frac{1}{2} \frac{d}{d s}\left\|\tilde{p}_{\varepsilon}\right\|_{L^{2}(\Omega)}^{2}+\frac{3}{4}\left\|\sqrt{a(N)} \nabla \tilde{p}_{\varepsilon}\right\|_{L^{2}(\Omega)}^{2}+\varepsilon\left\|\nabla \tilde{p}_{\varepsilon}\right\|_{L^{2}(\Omega)}^{2}\right)}\; ds  \leqslant \\
       \displaystyle\int_{0}^{T}{\frac{\alpha^{2}}{2}\left\|\tilde{q}_{\varepsilon}\right\|_{L^{2}(\Omega)}^{2}+\frac{\gamma_{N}^{2}}{2}\|N-N d\|_{L^{2}(\Omega)}^{2}}\;ds
       +\displaystyle\int_{0}^{T}(K+1)\left\|\tilde{p}_{\varepsilon}\right\|_{L^{2}(\Omega)}^{2}.
\end{multline}

\noindent Finally, we multiply \eqref{est q eps} by $\delta >0$ and add it to  \eqref{pesp estm} to obtain: 
$$
\begin{aligned}
&\displaystyle\int_{0}^{T}\frac{1}{2} \frac{d}{d s}\left(\delta\left\|\tilde{q}_{\varepsilon}\right\|_{L^{2}(\Omega)}^{2}+\left\|\tilde{p}_{\varepsilon}\right\|_{L^{2}(\Omega)}^{2}\right)+\frac{1}{2}\int_{0}^T\delta\left\|\nabla \tilde{q}_{\varepsilon}\right\|_{L^{2}(\Omega)}^{2}\; ds  \\
&\qquad+\int_{0}^T (\frac{3}{4}-\frac{\delta}{2} \norm{a(N)}_{L^{\infty}(Q_{T})}\norm{\mu(N)}_{L^{\infty}(Q_T)}^2)\left\|\sqrt{a(N)} \nabla \tilde{p}_{\varepsilon}\right\|_{L^{2}(\Omega)}^{2}\; ds
\displaystyle+\int_{0}^T \varepsilon\left\|\nabla \tilde{p}_{\varepsilon}\right\|_{L^{2}(\Omega)}^{2}\;ds  \\
&\leqslant \int_{0}^T{\left(\delta(\beta+\frac{1}{2}+\|f\|_{L^{\infty}(Q_T)})\left\|\tilde{q}_{\varepsilon}\right\|_{L^{2}(\Omega)}^{2}+\delta\frac{\gamma_{C}^{2}}{2}\|C-C_d\|_{L^{2}(\Omega)}^{2}+\frac{\alpha^2}{2}\left\|\tilde{q}_{\varepsilon}\right\|_{L^{2}(\Omega)}^{2}+\frac{\gamma_N^2}{2}\|N-N_d\|_{L^{2}(\Omega)}^{2}\right)}\;ds \\
&\qquad\displaystyle+\int_{0}^T{(K+1)\left\|\tilde{q}_{\varepsilon}\right\|_{L^{2}(\Omega)}^{2}}\;ds.
\end{aligned}
$$
We choose $\delta > 0$ such that $(\frac{3}{4}-\frac{\delta}{2} \norm{a(N)}_{L^{\infty}(Q_T)}\norm{\mu(N)}_{L^{\infty}(Q_T)}^2)$ is positive to ensure the positivity of the third term in the left-hand side in the above equation. Therefore, by Gronwall's Lemma, we deduce that $\tilde{p}_{\varepsilon}$ and $\tilde{q}_{\varepsilon}$ satisfy:
\begin{equation}\label{cv final p eps}
  \tilde{p}_{\varepsilon} \in L^{\infty}\left(0, T ; L^{2}(\Omega)\right),
\end{equation}
\begin{equation}\label{cv final  qeps}
\tilde{q}_{\varepsilon} \in L^{2}(0, T ; H^1(\Omega)),
\end{equation}
\begin{equation}\label{radical epsi]}
    \sqrt{\varepsilon} \nabla \tilde{p}_{\varepsilon} \in L^{2}\left(0, T ; L^{2}(\Omega)\right), %\quad  %\text{and}  
\end{equation}  
\begin{equation}\label{radical}
   \sqrt{a(N)} \nabla \tilde{p}_{\varepsilon} \in L^{2}\left(0, T ; L^{2}(\Omega)\right),
\end{equation}
\noindent and are bounded in these spaces independently of $\varepsilon$.\\
\noindent Thus, there exist a solution $ (\tilde{p},\tilde{q})$ and a subsequence of $(\tilde{p}_\varepsilon,\tilde{q}_{\varepsilon})$ still denotes as the sequence such that, as $\varepsilon$ goes to $0$,
\begin{align}
\tilde{p}_{\varepsilon} &\rightharpoonup \tilde{p} \mbox{ weakly* in } L^{\infty}\left(0, T ; L^{2}(\Omega)\right)  \, ,\label{cv p esp}\\
\tilde{q}_{\varepsilon} &\rightharpoonup \tilde{q} \mbox{ weakly in } L^{2}(0,T; H^{1}(\Omega)) \, ,\label{cv q eps}\\
\varepsilon \nabla \tilde{p}_{\varepsilon}=\sqrt{\varepsilon}(\sqrt{\varepsilon}\nabla \tilde{p}_{\varepsilon}) &\rightharpoonup 0 \mbox{ weakly in } L^{2}(Q_T), \, \label{cv espi eps}
\\
\sqrt{a(N)} \nabla \tilde{p}_{\varepsilon} &\rightharpoonup \xi \mbox{ weakly in } L^{2}(Q_T). \, \label{cv mhm}
\end{align}
%==================================%

\noindent Next, let $\phi \in  L^2(0,T,H^{1}(\Omega)$), we have 
$$
\begin{aligned}
 \left|\left\langle\partial_{s} \tilde{q}_{\varepsilon}, \phi\right\rangle \right| &\leqslant\left\|\nabla \tilde{q}_{\varepsilon}\right\|_{L^{2}(\Omega)}\norm{\nabla \phi}_{L^{2}(\Omega)} 
 +\norm{\sqrt{a(N)}}_{L^{\infty}(\Omega)}\norm{\mu(N)}_{L^{\infty}(\Omega)}\left\|\sqrt{a(N)} \nabla \tilde{p}_{\varepsilon}\right\|_{L^{2}(\Omega)}\|\nabla \phi\|_{L^{2}(\Omega)} \\
 &\quad+\beta\left\|\tilde{q}_{\varepsilon}\right\|_{L^{2}(\Omega)}\|\phi\|_{L^{2}(\Omega)}+\|f\|_{L^{\infty}(\Omega)}\left\|\tilde{q}_{\varepsilon}\right\|_{L^{2}(\Omega)}\|\phi\|_{L^{2}(\Omega)} 
 +\gamma_{C}\|C-C_d\|_{L^{2}(\Omega)}\|\phi\|_{L^{2}(\Omega)}.
\end{aligned}
$$
\noindent Now, using \eqref{cv final  qeps}, \eqref{radical} and Theorem \ref{th:existence}, we can deduce that
\begin{equation}
  \left\|\partial_{s} \tilde{q}_{\varepsilon}\right\|_{L^{2}(0,T;(H^{1}(\Omega))')} \leqslant A
\end{equation}
for some $A \geqslant 0$ independent of $\varepsilon$.

%$$
%\displaystyle\left|\left\langle\partial_s \tilde{q}_{\varepsilon}, \phi\right\rangle \right| \leqslant A\|\phi\|_{H^1(\Omega))},
%$$
%for some $A \geqslant 0$ independent of $\varepsilon$.
\noindent Therefore, up to subsequence, still denoted as $(\tilde{q}_{\varepsilon})_\varepsilon$,  we obtain
\begin{equation}\label{drive en tem q eps}
\frac{\partial \tilde{q}_{\varepsilon}}{\partial s} \rightharpoonup \frac{\partial \tilde{q}}{\partial s} \text { in } {L^{2}(0,T;(H^{1}(\Omega))')}.   
\end{equation}

\noindent On the other hand and for $\phi\in L^q(0,T,H^1(\Omega))$ with $q=\frac{2p}{p-2}$; $p>2$, Cauchy- Schwarz's inequality implies that
\begin{equation*}
\begin{split}
 \left|{ \left\langle\partial_s \tilde{p}_{\varepsilon}, \phi\right\rangle\; } \right|& \leqslant\norm{\sqrt{a(N)}}_{L^{\infty}(\Omega)}\left\|\sqrt{a(N)} \nabla \tilde{p}_{\varepsilon}\right\|_{L^{2}(\Omega)}\|\nabla\phi\|_{L^{2}(\Omega)}  +\left\|\sqrt{\varepsilon} \nabla \tilde{p}_{\varepsilon}\right\|_{L^{2}(\Omega)}\|\nabla\phi\|_{L^{2}(\Omega)}\\
&\quad+\kappa{\norm{\sqrt{a(N)}\nabla N}}_{L^{\infty}(\Omega)}\left\|\sqrt{a(N)} \nabla \tilde{p}_{\varepsilon}\right\|_{L^{2}(\Omega)}\|\phi\|_{L^{2}(\Omega)} \\
&\quad+(\kappa\norm{\mu(N)}_{L^{\infty}(\Omega)}\norm{\sqrt{a(N)}}_{L^{\infty}(\Omega)}\\
&\quad+\norm{\mu^{\prime}(N)}_{L^{\infty}(\Omega)}\norm{\sqrt{a(N)}}_{L^{\infty}(\Omega)})\|\nabla C\|_{L^{\infty}(\Omega)}\left\|\sqrt{a(N)} \nabla \tilde{p}_{\varepsilon}\right\|_{L^2(\Omega)}\|\phi\|_{L^{2}(\Omega)}\\
&\quad+\alpha\left\|\tilde{q}_{\varepsilon}\right\|_{L^{2}(\Omega)}\|\phi\|_{L^{2}(\Omega)}
 +\gamma_N\|N-N_d\|_{L^2(\Omega)}\|\phi\|_{L^2(\Omega)}.
%\\[0.1em]&\leqslant \displaystyle B\|\phi\|_ {H^1(\Omega))},
\end{split}
\end{equation*}
\noindent Using the estimates  \eqref{cv final  qeps}, \eqref{radical}, \eqref{radical epsi]}, Theorem \ref{th:existence} and the regularity results on $C$ (see \eqref{Cregularity}), we obtain
\begin{equation}
  \left\|\partial_{s} \tilde{p}_{\varepsilon}\right\|_{L^{q^{\prime}}(0,T;(H^{1}(\Omega))')} \leqslant B   
\end{equation}
\noindent for some $B \geqslant 0$ independent of $\varepsilon$ and $q^{\prime}=\frac{2p}{p+2}$. Hence,
up to subsequence, still denoted as $(\tilde{p}_{\varepsilon})_\varepsilon$,  we obtain
\begin{equation}\label{drive en tem p eps}
\frac{\partial \tilde{p}_{\varepsilon}}{\partial s} \rightharpoonup \frac{\partial\tilde{p}}{\partial s} \text { in } {L^{q^{\prime}}(0,T;(H^{1}(\Omega))')}.
\end{equation}

%===========================================================%
%==============================================================%
%==============================================================%
%\begin{remarque}
%We consider the limit $\xi$ of $\sqrt{a(N)} \nabla \tilde{p}_{\varepsilon}$ in \eqref{cv mhm}. As $N$ takes the values $1$ or $0$ at a finite number of points in $[0,T]\times\Omega$, this limit can be expressed as $\sqrt{a(N)} \nabla \tilde{p}$. The case where $N$ takes the values $1$ or $0$ at infinite  points in $[0,T]\times\Omega$ will be treated in the later.
%the task of finding the limit becomes challenging. In order to address this complexity, we will introduce a new concept of a weak formulation with enhanced regularity in the second section.
%\end{remarque}
%===============================================================================================================================%
%\subsubsection{Passing to the limit ($a(N)$ vanishes at a finite number of points in $[0,T]\times \Omega$)}

%\noindent Let us now tend the regularization parameter $\varepsilon$ to $0$ in the case where  $a(N)$ vanishes at a finite number of points in $[0,T]\times \Omega$. 

\noindent It remains to show that
\begin{equation}\label{hii}
    \xi=\sqrt{a(N)} \nabla \tilde{p} \mbox{ in }L^2(Q_T).
\end{equation}
Indeed, we consider  $\xi_{\varepsilon}^i=\sqrt{a(N)}\displaystyle \frac{\partial \tilde{p}_{\varepsilon}}{\partial x_i}, i=1, \ldots, N$ and we define 
 the open subset $Q_{T_\delta}:=\{(t,x)\in Q_{T}  ; a(N(t,x))>\delta\}$; $\delta >0$. Thus, $\xi_{\varepsilon}^i \rightharpoonup \xi^i$ weakly in $L^2(Q_{T_\delta})$. Moreover,  we have
$$
\frac{\partial \tilde{p}_{\varepsilon}}{\partial x_i}=\frac{1}{\sqrt{a(N)}} \xi_{\varepsilon}^i \rightharpoonup \frac{1}{\sqrt{a(N)}} \xi^i \text { weakly in } L^2\left(Q_{T\delta}\right) \text {, as } \varepsilon \rightarrow 0 .
$$
On the other hand, $\displaystyle\frac{\partial \tilde{p}_{\varepsilon}}{\partial x_i} \rightarrow \displaystyle\frac{\partial \tilde{p}}{\partial x_i}$ in the sense of distributions. We conclude that $\xi^i=$ $\sqrt{a(N)} \displaystyle\frac{\partial \tilde{p}}{\partial x_i}$ in $L^2(Q_{T_\delta})$ due to the uniqueness of the limit. Letting $\delta$ tend to $0$ and applying the dominated convergence theorem, we obtain $\xi^i=\sqrt{a(N)}\displaystyle \frac{\partial \tilde{p}}{\partial x_i}$ in $L^2(Q_T)$. 

\noindent With the above convergence, we now need to verify that the pair $(\tilde{p},\tilde{q})$ constitutes a weak solution to the modified degenerate adjoint problem \eqref{adjoint2}-\eqref{adj bord2}. For the sake of clarity, we write the  equality \eqref{phpepsi} in the following form: 
$$
K_1+K_2+K_3-K_4-K_5-K_6=0.
$$
Firstly, we take the test function $\varphi_1$ in  $ L^{\frac{2p}{p-2}}(0,T,H^1(\Omega))$ with  $p>2$. The weak convergence \eqref{drive en tem p eps} implies, when $\varepsilon$ goes to zero, that 
$$
K_{1} \longrightarrow \int_0^T < \frac{\partial \tilde{p}}{\partial s},\varphi_{1}> ds. 
$$ 
Next, 
\begin{equation*}
\begin{array}{ll}
\displaystyle K_2&=\displaystyle\iint_{Q_{T}}{ a(N) \nabla \tilde{p}_{\varepsilon} \cdot\nabla\varphi_1}\; d x d s+\iint_{Q_{T}}{\varepsilon \nabla \tilde{p}_{\varepsilon}\cdot \nabla \varphi_1}\; dx ds 
\medskip
\\&\displaystyle =\iint_{Q_{T}}{ \sqrt{a(N)} \nabla \tilde{p}_{\varepsilon}\cdot\left(\sqrt{a(N)} \nabla \varphi_1\right)}\; d x d s + \iint_{Q_{T}}{\sqrt{\varepsilon} \left(\sqrt{\varepsilon} \nabla \tilde{p}_{\varepsilon} \cdot\nabla \varphi_1\right)}\; dx ds.
\end{array}
\end{equation*}
Using \eqref{cv mhm}, \eqref{hii} and as $\sqrt{a(N)} \nabla \varphi_1 \in L^2(Q_T)$, we get 

$$\displaystyle \iint_{Q_{T}}{ \sqrt{a(N)} \nabla \tilde{p}_{\varepsilon}\cdot\left(\sqrt{a(N)}\nabla \varphi_1\right)}\; d x d s\underset{\varepsilon\to 0}{\longrightarrow} \iint_{Q_{T}}{ a(N) \nabla \tilde{p}\cdot \nabla  \varphi_1} \; dx ds \, .$$

\noindent \noindent
In addition to that, the assertion  \eqref{radical epsi]} leads to
$$
\sqrt{\varepsilon}\iint_{Q_{T}}{ \left(\sqrt{\varepsilon} \nabla \tilde{p}_{\varepsilon}\cdot \nabla \varphi_1\right)}\; dx ds  \underset{\varepsilon\to 0}{\longrightarrow}0 . 
$$ 

\noindent Therefore, when $\varepsilon$ tends to $0$,
$$\displaystyle K_{2} \longrightarrow \iint_{Q_{T}}{ a(N) \nabla \tilde{p}\cdot \nabla  \varphi_1} \; dx ds \, .$$
 %======================%
For the third term, and under the assumption that $a^\prime(N)\le \kappa\; a(N)$, we have 
\begin{equation*}
\begin{split}
    & \left| K_{3}-\iint_{Q_{T}}{ a'(N)\nabla N\cdot\nabla \tilde{p}\; \varphi_{1}}\; dx ds \right|\\
     &=\left|\iint_{Q_{T}}{ a'(N)\nabla N\cdot\nabla (\tilde{p}_{\varepsilon}-\tilde{p})\;\varphi_{1}}\; dx ds \right|
    \leqslant  \kappa \norm{\sqrt{a(N)}\nabla N}_{L^{\infty}(Q_{T})}\iint_{Q_{T}}{ \abs{ \sqrt{a(N)}\nabla (\tilde{p}_{\varepsilon}-\tilde{p})\;\varphi_{1}}\; dx ds\,.} 
    \end{split}
\end{equation*}
Using \eqref{cv mhm}, \eqref{regul gradient} and \eqref{hii}, the right-hand side term goes to $0$. Then, 
$$ 
K_{3} \underset{\varepsilon\to 0}{\longrightarrow} \iint_{Q_{T}}{ a'(N)\nabla N\cdot\nabla\tilde{p}\; \varphi_{1}}\; dx ds\, .
$$
%=====================%

\noindent Using similar guidelines, one has  $$\displaystyle K_{4} \underset{\varepsilon\to 0}{\longrightarrow} \iint_{Q_{T}}{ \chi'(N) \nabla\tilde{p}\cdot \nabla  \varphi_1} \; dx ds .$$
%==========%

\noindent Finally and due to $\eqref{cv final  qeps}$, $$ K_{5} \longrightarrow  \iint_{Q_{T}}{  \alpha \ \tilde{q} \ \varphi_1}  \; dx dt  ,\mbox{ as } \varepsilon \rightarrow 0 \, .$$ 
%================================================%
%===============================================================%

\noindent Furthermore, the equation \eqref{qhq} can be written as $V_1+V_2-V_3+V_4-V_5-V_6=0$. Let $\varphi_2$ in  $ L^{2}(0,T,H^1(\Omega))$. Then, using \eqref{cv final  qeps} and \eqref{cv final  qeps}, the convergence of each term, when $\varepsilon$ tends to $0$, is given as follow:
$$ V_{1} \longrightarrow \int_0^T < \frac{\partial \tilde{q}}{\partial s},\varphi_{2}> ds\, .$$ 
$$ V_{2} \longrightarrow  \iint_{Q_{T}}{  \nabla \tilde{q} \cdot\nabla\varphi_2}  \; dx ds \, .$$
\noindent For the third term $V_{3}$, we can follow the same proof as $K_{3}$ to obtain 
  $$ V_{3} \longrightarrow   \iint_{Q_{T}}{  \chi(N) \nabla \tilde{p}\cdot \nabla \varphi_2}  \; dx ds \, .$$
Next, when $\varepsilon$ tends to $0$,
$ \displaystyle V_{4} \longrightarrow  \iint_{Q_{T}}{  \beta\ \tilde{q}\ \varphi_2}  \; dx ds\, $ 
and
$\displaystyle V_{5} \longrightarrow  \int_0^T\!\!\!\! \int_{\Omega_{c}}{  f\ \tilde{q} \ \varphi_2}  \; dx ds\, . $

\noindent Hence, the existence of a solution to \eqref{php}-\eqref{qhq} is well achieved. 
%=========================================%
%===============================================================%
%==============================================================%
%==========================================================%
%==============================================================%

\subsubsection{Convergence under  the main assumption \eqref{aa}}
In this subsection, we consider that the function $a(N)$ satisfies the hypothesis \eqref{aa}. Here, the difficulty is to ensure the convergence of the sequence $(\sqrt{a(N)}\nabla \tilde{p}_\varepsilon)_\varepsilon$. The estimate \eqref{radical} guarantees the weak convergence to a function $\xi$ but the identification of the limit is not obvious, as in the previous case. For that, 
%Since this convergence remains elusive, we are confronted with a formidable issue.
%To tackle this convergence problem, 
we seek a weaker alternative variational formulation with more regular test functions and with further assumptions on the cell density $N$. More precisely,  additional regularity assumptions are given as follows: 
%\begin{equation}\label{chi1}
 %\chi(N)= \mu(N)a(N)\in \mathcal{C}^{2}[0,1], \quad  \forall N\in[0,1]
%\end{equation}
%and 
\begin{equation}\label{a(N)}
\hspace{-6.5em} a(N), \chi(N) \in  L^{2}(0,T;H^{2}( \Omega)) \, .
\end{equation}
Next, we define $ \tilde{A}(N):= \displaystyle \int_0^N a'(r) dr.$ Now, let us introduce the following weak problem:

%==========================================%
\begin{definition}\label{def adj2}
%Assume that $\displaystyle \tilde{A}(N):= \int_0^N a'(r) dr \in L^{2}(0,T;H^{2}( \Omega))$.
A pair  $(\tilde{p},\tilde{q})$ is a {\it{\bf weak solution}} of \eqref{adjoint2}-\eqref{adj bord2} if
\begin{align*}
&\tilde{p} \in L^{\infty}(0,T;L^{2}( \Omega)) , \, \quad \tilde{q}\in L^{\infty}(0,T;L^{2}( \Omega)) \cap L^{2}(0,T;H^{1}( \Omega)) , \\
&\partial_{s}\tilde{p}\in L^{\frac{2p}{p+2}}(0,T;(H^{1}( \Omega))^{\prime}), \quad \partial_{s}\tilde{q}\in L^{2}(0,T;(H^{1}( \Omega))^{\prime}),
\end{align*}
and $(\tilde{p},\tilde{q}$) satisfy
\begin{equation}\label{php2}
\begin{split}
  &\displaystyle{ \int_{0}^T {<\partial_{s}\tilde{p}, \varphi_1>}\; dx ds - \iint_{Q_{T}}{ \tilde{p}\ \big( \nabla a(N) \cdot \nabla \varphi_1+a(N) \Delta \varphi_1\  \big)} \; dx ds-\iint_{Q_{T}}{ \tilde{p} \nabla \varphi_1 \cdot \nabla \tilde{A}} \; dx ds}\\
   &\displaystyle{ -\iint_{Q_{T}}{ \tilde{p}\  \varphi_1  \Delta\tilde{A}} \; dx ds+ \iint_{Q_{T}}{ \tilde{p} \big(\  \nabla (\chi'(N) \varphi_1)\cdot \nabla C +\chi'(N) \varphi_1 \Delta C \ \big)} \; dx ds}\\
   &\displaystyle{- \iint_{Q_{T}}{ \alpha\ \tilde{q} \varphi_1 } \; dx ds- \iint_{Q_{T}}{ \gamma_N(N-N_d)\varphi_1 } \; dx ds=   0}\, , 
\end{split}
\end{equation}
\begin{equation}\label{qhq2}
\begin{split}
& \int_{0}^T {\partial_{t}\tilde{q}, \varphi_2}\; dx ds+ \iint_{Q_{T}}{\nabla \tilde{q}\cdot \nabla \varphi_2} \; dx ds+\iint_{Q_{T}}{\tilde{p}\ \big( \nabla \chi(N) \cdot \nabla \varphi_2+\chi(N) \Delta \varphi_2\  \big)}\; dxds+\iint_{Q_{T}}{ \beta \tilde{q} \varphi_2 } \; dx ds\,\\
& 
 -\iint_{Q_{T}}{\gamma_C(C-C_d) \varphi_2 } \; dx ds\,=0 \, ,   
\end{split}
\end{equation}
for all $\varphi_{1} \in C^{0}(0,T;H^{2}(\Omega))$ and $\varphi_{2}  \in L^{2}(0,T;H^{2}(\Omega))$ .
\end{definition}
%======================================
\noindent The main result of this subsection is therefore provided through the following theorem:

\begin{theoreme}\label{th exst adj2}
Assume that $ a'(N)\leqslant \kappa\, a(N)$; $\kappa>0$, that there exists a function $\mu \in \mathcal{C}^{2}\left([0,1] ; \mathbb{R}^{+}\right)$;  $\mu(u)=\frac{\chi(u)}{a(u)}$, $\forall u \in(0,1)$ and $\mu(0)=\mu(1)=0$ and  that    \eqref{regul gradient} and \eqref{a(N)} hold. Moreover, let (N, C) be a
weak solution of system \eqref{CNS}-\eqref{ci}. Therefore, if  the hypotheses of Theorem \ref{th:existence} hold, then there exists a  solution to the modified adjoint system \eqref{adjoint2}-\eqref{adj bord2} 
 in the sense of Definition \ref{def adj2}.
\end{theoreme}
%
%
%=============================================%
\begin{proof}
\noindent We consider initially the problem \eqref{phpepsi}-\eqref{qhqepsi}, which has at least one solution ($p_\varepsilon,q_\varepsilon)$, but with more regular test functions. Thus, $(p_\varepsilon,q_\varepsilon)$ satisfies the following weak non-degenerate problem:\\ 

%For  the non-degenerate problem, we have the following definition:
%\begin{definition}%\label{prop adj exis}
%   A pair  $(\tilde{p}_{\varepsilon},\tilde{q}_{\varepsilon})$ is a {\bf weak solution} of \eqref{adj non dege}-\eqref{adj non  bord }, if
%\begin{align*}
\noindent Find $(\tilde{p}_{\varepsilon},\tilde{q}_\varepsilon )  \in L^{\infty}(0,T;L^{2}( \Omega)) \times  \left(L^{\infty}(0,T;L^{2}( \Omega)) \cap L^{2}(0,T;H^{1}( \Omega))\right)$ such that for all $\varphi_{1} \in C^{0}(0,T;H^{2}(\Omega))$ and  $\varphi_{2}  \in L^{2}(0,T;H^{2}(\Omega))$, we have:
%\end{align*}
%\noindent and $(\tilde{p}_{\varepsilon},\tilde{q}_{\varepsilon}$) satisfy     
\begin{equation}\label{phpepsi2}
\begin{array}{ll}
\medskip
\displaystyle{ \int_{0}^T {<\partial_{s}\tilde{p}_\varepsilon, \varphi_1>}\; dx ds - \iint_{Q_{T}}{ \tilde{p}_\varepsilon\ \big( \nabla a_\varepsilon(N) \cdot \nabla \varphi_1+a_\varepsilon(N) \Delta \varphi_1\  \big)} \; dx ds-\iint_{Q_{T}}{ \tilde{p}_\varepsilon \nabla \varphi_1 \cdot \nabla \tilde{A}} \; dx ds}\\
\medskip
\displaystyle{ -\iint_{Q_{T}}{ \tilde{p}_\varepsilon\  \varphi_1  \Delta\tilde{A}} \; dx ds+ \iint_{Q_{T}}{ \tilde{p}_\varepsilon \big(\  \nabla (\chi'(N) \varphi_1)\cdot \nabla C +\chi'(N) \varphi_1 \Delta C \ \big)} \; dx ds}\\

\displaystyle{- \iint_{Q_{T}}{ \alpha\ \tilde{q}_\varepsilon \varphi_1 } \; dx ds- \iint_{Q_{T}}{ \gamma_N(N-N_d)\varphi_1 } \; dx ds=   0} , 
\end{array}
\end{equation}

\begin{equation}\label{qhqepsi2}
\begin{array}{ll}
\medskip
\displaystyle\int_{0}^T {<\partial_{s}\tilde{q}_\varepsilon, \varphi_2>}\; dx ds+ \iint_{Q_{T}}{\nabla \tilde{q}_\varepsilon\cdot \nabla \varphi_2} \; dx ds+\iint_{Q_{T}}{ \tilde{p}_\varepsilon\ \big( \nabla \chi(N) \cdot \nabla \varphi_2+\chi(N) \Delta \varphi_2\  \big)\; dxds}\\
\displaystyle+\iint_{Q_{T}}{ \beta \tilde{q}_\varepsilon \varphi_2 } \; dx ds\,
\displaystyle -\iint_{Q_{T}}{\gamma_C(C-C_d) \varphi_2 } \; dx ds\,=0 ,  
\end{array}
\end{equation}
\\

%\end{definition}
%===============================================%
%===================================================%
%=============================================================%

\noindent We recall that the solution $(\tilde{p}_{\varepsilon},\tilde{q}_\varepsilon )$ satisfies  the convergences properties \eqref{cv p esp}-\eqref{cv mhm}, \eqref{drive en tem q eps} and \eqref{drive en tem p eps}. Consequently, under the assumptions \eqref{regul gradient} and \eqref{a(N)}, transitioning to the limit becomes straightforward as follows:\\
%==============================================================%
%=================================================================%
%===================================================================%

%\noindent With the above convergence \eqref{} to established in the previous section, we still need to verify that the pair $(\tilde{p},\tilde{q})$ constitutes a weak solution to the degenerate adjoint problem \eqref{adjoint2}.

\noindent First, we write the equality \eqref{phpepsi2} as $K_1 - K_2 - K_3 - K_4 + K_5 - K_6 - K_7 = 0$. Then, assertion \eqref{drive en tem p eps} implies that  
$$ K_{1} \longrightarrow \int_0^T < \frac{\partial \tilde{p}}{\partial t},\varphi_{1}> ds, \mbox{ as } \varepsilon \mbox{ tends to zero. }$$ 
%============================================================%
%==================================================================%

\noindent For  the second term, we have
\begin{equation*}
\begin{split}
 &\abs{K_2-{ \iint_{Q_{T}}{ \tilde{p}\big( \nabla a(N) \cdot \nabla \varphi_1+a(N) \Delta \varphi_1 \big)}} \; dx ds}\\
 &=\abs{\iint_{Q_{T}}{ \tilde{p}_{\varepsilon}\big( \nabla a_{\varepsilon}(N) \cdot \nabla \varphi_1+a_{\varepsilon}(N) \Delta \varphi_1 \big)} \; dx ds-{ \iint_{Q_{T}}{ \tilde{p}\big( \nabla a(N) \cdot \nabla \varphi_1+a(N) \Delta \varphi_1 \big)}} \; dx dt}\\
&\leq \iint_{Q_{T}}{ \abs{(p_{\varepsilon}-p) \nabla a(N) \cdot \nabla \varphi_1}}  dx dt+\norm{a(N)}_{L^{\infty}}\iint_{Q_{T}}{\abs{(\tilde{p}_{\varepsilon}-\tilde{p})  \Delta \varphi_1}}  dx dt+\iint_{Q_{T}}{ \abs{ \varepsilon \tilde{p}_{\varepsilon} \Delta \varphi_1}}dx dt \, .
\end{split}
\end{equation*}

\noindent Using the assumptions \eqref{regul gradient}, \eqref{aa}, and the regularity of the test function, we have $\nabla a(N)\cdot\nabla \varphi \in L^2(Q_{T})$. Consequently, by using \eqref{cv p esp} and \eqref{radical epsi]}, we conclude that

$$\abs{K_2-{ \iint_{Q_{T}}{ \tilde{p}\big( \nabla a(N) \cdot \nabla \varphi_1+a(N) \Delta \varphi_1 \big)}} \; dx ds}\underset{\varepsilon\to 0}{\longrightarrow} 0 \, .$$%=======================================%

\noindent Next, using the assumption \eqref{a(N)} and the regularity of the test function $\varphi_1$,
%$H^1(\Omega)\hookrightarrow L^{4}(\Omega)$,
we have $\nabla \varphi_{1}\cdot\nabla \tilde{A}\in L^{2}(Q_{T})$. Therefore, by \eqref{cv p esp},
 $$\abs{K_3-\iint_{Q_{T}}{  \tilde{p} \nabla \varphi_1 \cdot  \nabla \tilde{A}} \; dx ds}\underset{\varepsilon\to 0}\longrightarrow 0 \, .$$ 
%=====================%

\noindent For the fourth term, using the Sobolev injection 
%(see lemma \ref{Sobolev}) 
and the assumption \eqref{a(N)}, we get $\varphi_1 \Delta\tilde{A}\in L^2(Q_{T})$. Then
 $$ K_{4} \longrightarrow \iint_{Q_{T}}{  \tilde{p}\; \varphi_1 \Delta\tilde{A}} \; dx ds \, , \mbox{ as $\varepsilon$ tends to $0$\, .}$$
%==========%

\noindent Similarly, using assumptions \eqref{a(N)}, Theorem \ref{th:existence} and the regularity of the test function, we can conclude that 
$${K5\longrightarrow \iint_{Q_{T}}{ \tilde{p}\big(\  \nabla (\chi'(N) \varphi_1)\cdot \nabla C+\chi'(N) \varphi_1 \Delta C \ \big)} \; dx ds} \, , \mbox{ as $\varepsilon$ tends to $0$\, .}$$

\noindent Then,  using \eqref{cv final  qeps}, 
$$\hspace{-7em} K_{6} \longrightarrow  \iint_{Q_{T}}{  \alpha \ \tilde{q} \ \varphi_1}  \; dx ds \, , \mbox{ as $\varepsilon$ tends to $0$\, .}$$ 

\noindent Finally, we conclude that the  weak system \eqref{phpepsi2}-\eqref{qhqepsi2} has at least one solution.

\end{proof}

%=======================Conclusion=========================%
\section{Conclusion}

In this paper, a controlled chemotaxis model, consisting of parabolic-parabolic equations with degenerate diffusion and a control element representing the concentration coefficient, has been explored. The main challenges arise from the significant degeneracy of the diffusive term, which impacts the existence of weak solutions for both the direct and adjoint problems, as well as the consistent attainment of the maximum principle. 
Due to the wide range of applications in biology and the lack of weak solutions investigations in recent related papers, this study has provided rigorous answers to various theoretical problems. The well-posedness of the direct problem, the first-order optimality conditions, and the existence of weak solutions to the dual problem have been well-detailed.  The discretization of the direct and dual problems is one of the perspectives of this paper. The numerical study will be the cornerstone of many algorithms that solve inverse problems. To sum up,  this comprehensive approach will provide deeper insights into the system's dynamics and facilitate practical applications across various fields.

%=========================================================================================%
\end{quotation}

% BibTeX users please use
%\bibliographystyle{spmpsci}
\bibliographystyle{abbrv}
%\begin{thebibliography}{9}
%\bibliography{sample}  % remember to edit the file name

\end{document}